\documentclass[a4paper, 12pt]{article}

\usepackage[utf8]{inputenc}
\usepackage[margin=2.5cm]{geometry}

\usepackage{amsmath, amsthm, amssymb}
\usepackage{array}
\usepackage{graphicx}
\usepackage{enumerate}
\usepackage{authblk}
\usepackage{thm-restate}
\usepackage{float}

\usepackage{tikz}

\usepackage[hypertexnames=false, colorlinks=true, linkcolor =blue!60!black, citecolor=orange!70!black]{hyperref}

\usepackage[capitalise, nameinlink]{cleveref}
\Crefname{observation}{Observation}{Observations}

\declaretheorem[name=Theorem, numberwithin=section]{theorem}
\declaretheorem[name=Lemma, sibling=theorem]{lemma}
\declaretheorem[name=Proposition, sibling=theorem]{proposition}

\declaretheorem[name=Corollary, sibling=theorem]{corollary}

\declaretheorem[name=Claim, sibling=theorem]{claim}
\declaretheorem[name=Remark, style=remark, sibling=theorem]{remark}
\declaretheorem[name=Observation, style=remark, sibling=theorem]{observation}
\declaretheorem[name=Question, style=remark, sibling=theorem]{question}

\def\cqedsymbol{\ifmmode$\lrcorner$\else{\unskip\nobreak\hfil
\penalty50\hskip1em\null\nobreak\hfil$\lrcorner$
\parfillskip=0pt\finalhyphendemerits=0\endgraf}\fi} 
\newcommand{\cqed}{\renewcommand{\qed}{\cqedsymbol}}

\newcommand{\ceil}[1]{\left \lceil #1 \right \rceil}

\DeclareMathOperator{\clin}{\chi_\mathrm{lin}}
\DeclareMathOperator{\ccen}{\chi_\mathrm{cen}}
\DeclareMathOperator{\td}{\mathbf{td}}
\DeclareMathOperator{\tw}{\mathbf{tw}}

\newcommand{\grid}{\boxplus}
\newcommand{\N}{\mathbb{N}}
\DeclareMathOperator{\Oh}{\mathcal{O}}

\let\le\leqslant
\let\ge\geqslant
\let\leq\leqslant
\let\geq\geqslant

\title{Linear colorings of graphs}

\author[1]{Claire Hilaire}
\author[2]{Matjaž Krnc}
\author[2]{Martin Milanič}
\author[3]{Jean-Florent Raymond}
\affil[1]{Université Clermont Auvergne, Clermont Auvergne INP, LIMOS, Clermont-Ferrand, France}
\affil[2]{FAMNIT and IAM, University of Primorska, Koper, Slovenia}
\affil[3]{CNRS, ENS de Lyon, Université Claude Bernard Lyon 1, LIP UMR 5668,
  Lyon, France}

\date{March 2026}

\begin{document}

\maketitle

\begin{abstract}
Motivated by algorithmic applications, Kun, O'Brien, Pilipczuk, and Sullivan introduced the parameter linear chromatic number as a relaxation of treedepth and proved that the two parameters are polynomially related. 
They conjectured that treedepth could be bounded from above by twice the linear chromatic number.

In this paper we investigate the properties of linear chromatic number and provide improved bounds in several graph classes.
\end{abstract}

\section{Introduction}

\begin{sloppypar}
Treedepth is a graph parameter that has been studied under different names, such as vertex ranking number \cite{schaffer1989optimal}, ordered coloring \cite{katchalski1995ordered}, and elimination height \cite{pothen1988complexity}. 
It plays a central role in the graph sparsity theory of Ne\v{s}et\v{r}il and Ossona de Mendez~\cite{nevsetvril2012sparsity} due to its strong ties with the notion of bounded expansion. 
Like several other concepts from sparsity theory, a convenient aspect of this parameter is that it has several equivalent characterizations, so one can choose the most relevant one depending on the context.
\end{sloppypar}

The definition of treedepth that we mostly consider in this work is via centered colorings. 
A \emph{centered coloring} of a graph $G$ is an assignment of integers (referred to as \emph{colors}) to its vertices such that in every connected subgraph of $G$ there is a vertex with a unique color, called \emph{center}. 
Observe that every centered coloring is proper (i.e., adjacent vertices receive different colors).
The minimum number of colors required in a centered coloring of $G$ is its \emph{centered chromatic number}, denoted by $\ccen(G)$.
As it turns out, this number is equal to the treedepth of $G$ (see \cite{nevsetvril2012sparsity}).

Kun, O'Brien, Pilipczuk, and Sullivan studied in \cite{kun2021polynomial} the following relaxation of centered colorings due to Felix Reidl and Fernando S\'anchez-Villaamil.
A \emph{linear} coloring of a graph $G$ is an assignment of integers to its vertices such that in every path of $G$ there is a vertex with a unique color. 
The \emph{linear chromatic number} of $G$ is the minimal number of colors required in a linear coloring of $G$, and we denote it by $\clin(G)$. 
As paths of $G$ are connected subgraphs, every centered coloring is a linear coloring, so we always have $\clin(G)\leq \ccen(G)$.
Observe that, just like centered colorings, every linear coloring is also proper.
The motivation of the authors of \cite{kun2021polynomial} for introducing the linear chromatic number is algorithmic.
They suggested an approach towards improving the running time (in practice) of an algorithmic pipeline for classes of bounded expansion~\cite{DRRSSS14,DKT13,NO08} and were able to reduce the question whether their algorithmic strategy is relevant to the following graph theory problem: how big can $\ccen$ be compared to $\clin$?
In the same paper they provided the following partial answer: there is a constant $c\leq 190$ such that for every graph $G$, $\ccen(G)\leq \clin(G)^c \cdot (\log \clin(G))^{\Oh(1)}$.

Using the recent results of \cite{czerwinski2021improved} the upper bound on $c$ could be improved to 19 (as observed in~\cite{kun2021polynomial}), and \cite{bose2022linear} further improved it to 10.

\begin{theorem}[\cite{bose2022linear}]
\label{th:polybd}
Every graph $G$ satisfies $\ccen(G)\leq \clin(G)^{10} \cdot (\log \clin(G))^{\Oh(1)}$.
\end{theorem}

On the other hand, the authors of \cite{kun2021polynomial} provided a sequence of $P_5$-free chordal graphs for which the ratio $\ccen(G)/\clin(G)$ can be arbitrarily close to $2$.
They conjectured that this could be the optimal gap between the two invariants.

\begin{restatable}[\cite{kun2021polynomial}]{conjecture}{conj}
\label{conj:main}
    For every graph $G$, $\ccen(G) \leq 2 \clin(G)$.
\end{restatable}

At the time of writing, the best bound that we are aware of on $\ccen$ in terms of $\clin$ in the general case is the aforementioned double-digit degree polynomial, quite far from the conjectured linear bound.
To the best of our knowledge, the only graph classes where a linear bound is known to hold between the two parameters are trees of bounded degree and pseudogrids (see Bose et al.~\cite{bose2022linear}), where a $k\times k$ pseudogrid is a subgraph-minimal graph containing the $k\times k$ grid as a minor.\footnote{The result from \cite{bose2022linear} that the relation between the linear and the centered chromatic numbers of pseudogrids is linear is only stated, but not entirely proved. 
What is proved in \cite{bose2022linear} is that the linear chromatic number of a $k\times k$ pseudogrid is in $\Omega(k)$.
For completeness, we prove that the centered chromatic number of a $k\times k$ pseudogrid is in $\mathcal{O}(k)$ in \Cref{sec:pseudogrids}.}

\begin{theorem}[{\cite[Theorem~4]{kun2021polynomial}}]\label{th:treedeg}
    Let $T$ be a tree of maximum degree $\Delta\geq 3$.
    Then $\ccen(T) \leq (\log \Delta)\cdot \clin(T)$.
\end{theorem}

In the case of interval graphs, \cite{kun2021polynomial} provided the following improvement over the general bound.
\begin{theorem}[\cite{kun2021polynomial}]\label{th:intv}
For every interval graph $G$, $\ccen(G) \leq \clin(G)^2$.
\end{theorem}

\subsection*{Our contribution}

This paper is an exploratory work on the topic of linear colorings. First, we investigate graph classes where the general polynomial bound given by \Cref{th:polybd} can be substantially improved. 
We give quadratic bounds for any minor-closed graph class, chordal graphs, and circular-arc graphs. 
The latter two results generalize the bound of \Cref{th:intv} for interval graphs. 
In the case of graphs of bounded treewidth, we provide a linear bound. 
This translates to a bound of 3.7 on the ratio $\ccen(T)/\clin(T)$ for every tree $T$ (Theorem~\ref{thm:trees}), improving the non-constant $\log \Delta(T)$ of Theorem~\ref{th:treedeg}. 
The proofs of these bounds rely on a combination of the results of \cite{czerwinski2021improved} on unavoidable subgraphs of graphs of large treedepth (see~\Cref{th:wd}) and a result of \cite{bose2022linear} about the linear chromatic number of pseudogrids (see~\Cref{th:lineargrid}).

We confirm \Cref{conj:main} in several restricted graph classes such as caterpillars (for which we show that the two parameters differ by at most one) and several graph classes for which we can show that linear and centered colorings coincide: $(P_3+P_1)$-free graphs, (claw, net)-free graphs (which contain in particular the proper interval graphs), complete multipartite graphs, and complements of rook's graphs.\footnote{See \cref{sec:lin=cen} for definitions.}
This is done by a close inspection of the structure of the considered graph classes.
Note that the result for proper interval graphs cannot be generalized to the class of interval graphs, as shown by the caterpillars.
Let us also remark that if there is a constant $c>0$ such that $\ccen(G)\le c\cdot \clin(G)$ for every chordal graph, then $c\ge 2$ due to the aforementioned construction from~\cite{kun2021polynomial}.
Our results in this direction are summarized in \Cref{table:sum}.
\begin{table}
\centering
\begin{tabular}{lll}
 Graph class & Upper bound for $\ccen(G)$ $\quad$ & Reference\\
 \hline
 any minor-closed graph class $\quad$ & $\Oh(\clin(G)^2)$& \Cref{th:minclo} \\
 chordal, circular-arc & $\Oh(\clin(G)^2)$& \Cref{cor:chordal}\\
 $G$ with $\tw(G)\leq t$ & $3.7\cdot(t+1) \cdot (\clin(G) + 1/\log 3)$ & \Cref{th:lineartw}\\
 trees  & $3.7\cdot \clin(G)$ & \Cref{thm:trees}\\
 caterpillars & $\clin(G) + 1$ & \Cref{th:cat}\\
 $(P_3+P_1)$-free graphs & $\clin(G)$ & \Cref{th:p3p1}\\
 (claw,net)-free & $\clin(G)$ & \Cref{th:CN-free}\\
 complete multipartite graphs & $\clin(G)$ & \Cref{compbip}\\
 co-rook's graphs & $\clin(G)$ & \Cref{corook}
\end{tabular}
\caption{Summary of our results regarding  \Cref{conj:main}.}
\label{table:sum}
\end{table}

The fact that the centered and the linear chromatic numbers coincide for the class of complete multipartite graphs follows from the analogous result for cographs due to Kun et al.~\cite{kun2021polynomial}.
However, for complete multipartite graphs, as well as for complements of rook's graphs, we determine the exact values of these parameters.

We then turn to another aspect of the linear chromatic number: obstructions. 
The class of graphs with linear chromatic number at most some constant $k$ is closed under subgraphs, so we can define its \emph{obstruction set} (with respect to the subgraph relation) as the set of subgraph-minimal elements that do not belong to the class.
Obstruction sets of graph classes closed under a certain graph containment relation are interesting, especially when they have bounded size, as they provide characterizations in terms of forbidden substructures in the spirit of Kuratowski's theorem and, in case of the subgraph relation, immediately lead to polynomial-time recognition algorithms.
In the case of the centered chromatic number, obstruction sets can be defined similarly. 
They have been shown to have finite size by Dvo{\v{r}}{\'a}k,  Giannopoulou, and Thilikos \cite{dvorak2012forbidden} (lower bounds can be found in the same paper).
 We observe that this property also holds for obstructions for bounded linear chromatic number (Corollary~\ref{cor:finite-subgraph}).
In \cite{dvorak2012forbidden} the set of obstructions for centered chromatic number at most $k$ is given for $k\in \{1,2,3\}$. (See also \cite{kuhn2025computing} for complements about obstructions with respect to the induced subgraph relation.)
By revisiting their proof, we provide obstruction sets for linear chromatic number at most $k$, for $k\in \{1,2,3\}$.

We conclude with algorithmic aspects. 
It is known that computing the centered chromatic number is NP-complete \cite{MR1642971}. We show that this also holds for computing the linear chromatic number (\Cref{th:npc}). Finally, we provide an FPT algorithm for deciding if a graph has linear chromatic number at most~$k$ (\Cref{th:fpt}).

\paragraph{Organization of the paper.}
In \Cref{sec:prelim} we introduce the necessary terminology and survey the basic properties of linear chromatic number. 
\Cref{sec:general} is devoted to the proofs of improved bounds for minor-closed graph classes, graphs of bounded treewidth, trees, and certain intersection graphs. In \Cref{sec:lin=cen} we investigate the classes where the two chromatic numbers are equal or differ by at most one. In \Cref{sec:obs} we discuss subgraph-obstructions to bounded linear chromatic number, with explicit obstruction sets for the values $1,2,3$. Algorithmic results are proved in \Cref{sec:algo}. We conclude in \Cref{sec:concl}.

\section{Preliminaries}
\label{sec:prelim}
\subsection{Notations and definitions}

\paragraph{Basics.}

In this paper logarithms are binary. 
Unless stated otherwise, we use standard graph theory terminology. The \emph{clique number} of a graph $G$ is the maximum order of a clique in $G$ and is denoted by $\omega(G)$. 
The \emph{complement} $\overline{G}$ of $G$ is the graph obtained by reversing the adjacency relation. 
For a subset $X\subseteq V(G)$, we use $G-X$ to refer to the graph obtained from $G$ after the deletion of the vertices in $X$.
A graph is \emph{subcubic} if its maximum degree is at most $3$.

For every $k\in \N$, $P_k$ is the path on $k$ vertices.
We denote by $P_3+P_1$ the disjoint union of the paths $P_3$ and $P_1$ and by \emph{paw} the complement of $P_3+P_1$.
The \emph{$k\times k$ grid} is the Cartesian product of two copies of $P_k$ and we denote it by $\grid_k$.
The complete binary tree with $k$ levels is denoted by~$B_k$.
A \emph{star} is any tree with a vertex adjacent to all other vertices.
A \emph{star forest} is a graph every component of which is a star.
A graph is \emph{chordal} if it has no induced cycle of length at least $4$.

\paragraph{Graph containment relations.}
We say that a graph $H$ is a \emph{minor} of a graph $G$ if a graph isomorphic to $H$ can be obtained from a subgraph of $G$ by contracting edges.
A graph class is \emph{minor-closed} if for every graph $G$ in the class, every minor of $G$ is also in the class, and it is a \emph{proper} minor-closed class if it is not the class of all graphs.
Given two graphs $G$ and $H$, we say that $G$ \emph{contains} $H$ if $H$ is isomorphic to a (not necessarily induced) subgraph of $G$, and that $G$ is \emph{$H$-free} if no induced subgraph of $G$ is isomorphic to $H$ and, more generally, that $G$ is $(H_1,H_2)$-free if no induced subgraph of $G$ is isomorphic to one of $H_1$ and $H_2$.
A graph $G$ is a \emph{subdivision} of a graph $H$ if $G$ is obtained from $H$ by subdividing edges.
A graph $H$ is a \emph{topological minor} of a graph $G$ if $G$ contains a subdivision of $H$.

\paragraph{Tree decompositions.}
A \emph{tree decomposition} of a graph $G$ is a pair $(T, \{X_t\}_{t\in V(T)})$ where $T$ is a tree and $X_t\subseteq V(G)$ for every $t\in V(T)$, with the following properties: $V(G) = \bigcup_{t\in V(T)}X_t$; for every $u\in V(G)$, the set $\{t\in V(T) \mid u \in X_t\}$ induces a connected subgraph of $T$; and for every edge $uv\in E(G)$ there is a $t\in V(T)$ such that $\{u,v\}\subseteq X_t$. The $X_t$'s are the \emph{bags} of the decomposition. The maximum size of a bag minus one is the \emph{width} of the decomposition. The minimum width of a tree decomposition of $G$ is its \emph{treewidth}, which we denote by $\tw(G)$.

\paragraph{Colorings.}
A \emph{graph coloring} is a function mapping the vertices of a graph to integers, referred to as \emph{colors}. 
A coloring is \emph{linear} if every path has a unique color and \emph{centered} if every connected subgraph has a unique color. The minimum number of colors (i.e., the cardinality of the image) of a linear (resp.\ centered) coloring of a graph $G$ is its \emph{linear chromatic number} (resp.\ \emph{centered chromatic number}), and we denote it by $\clin(G)$ (resp.\ $\ccen(G)$). As noted in the introduction, $\ccen$ is the parameter most often called \emph{treedepth} (see \cite[Proposition~6.6]{nevsetvril2012sparsity}).
%
%
%

\subsection{Centered versus linear}

In this section we survey the basic properties of the linear chromatic number. 
Not all the results stated in this section are used later in the paper, but they help to understand the main differences between linear and centered colorings.

From the definition we have the following inequality.
\begin{remark}
 \label{rem:weaker}
For every graph $G$, $\clin(G) \leq \ccen(G).$
\end{remark}

It is well known that the centered chromatic number does not increase when taking minors (see for instance \cite[Lemma~6.2]{nevsetvril2012sparsity}).
\begin{proposition}\label{prop:minorclosed}
    If $H$ is a minor of $G$, then $\ccen(H) \leq \ccen(G)$.
\end{proposition}

What about the linear chromatic number?
A simple observation shows that it does not increase when taking subgraphs.

\begin{observation}\label{obs:monolin}
    If $H$ is a subgraph of $G$, then $\clin(H) \leq \clin(G)$.
\end{observation}

\begin{proof}
    Every path of $H$ is a path of $G$. So the coloring of $H$ induced by a coloring of $G$ witnessing $\clin(G)$ is a linear coloring of $H$.
\end{proof}

We can also easily get the following by reserving a color for the deleted vertex.
\begin{observation}\label{obs:lininduced}
    If $H$ is an induced subgraph of $G$ obtained by removing a single vertex, then $\clin(G) \leq \clin(H)+1$.
\end{observation}

What happens with minors? 
\cref{obs:monolin,obs:lininduced} imply that contracting an edge cannot increase the linear chromatic number by more than one. 
A further partial answer is given by the following.

\begin{observation}
There is a polynomial $f$ such that if $H$ is a minor of $G$, then $\clin(H) \leq f(\clin(G))$.
\end{observation}

\begin{proof}
By \Cref{rem:weaker}, \Cref{prop:minorclosed}, and \cref{th:polybd}, respectively, we have $\clin(H)\leq \ccen(H) \leq \ccen(G) \leq f(\clin(G))$, where $f$ is a polynomial of degree $11$ whose existence is implied by \cref{th:polybd}.
\end{proof}

However, in general it is not true that the linear chromatic number is monotone under taking minors (in fact, not even under taking topological minors), as shown by the following construction due to Jana Masa\v{r}\'ikov\'a and Wojciech Nadara (personal communication, 2024). We note that \Cref{conj:main}, if true, would imply that $\clin(H)\leq 2\cdot \clin(G)$ for every minor $H$ of $G$.

\begin{proposition}\label{prop:mn}
    There is a graph $G$ with $\clin(G) = 4$ and  a topological minor $H$ of $G$ with $\clin(H)>4$.
\end{proposition}

\begin{proof}
Let $G$ be the graph from \cref{fig:nadara} (left) and $\psi$ be the coloring of $G$ specified by the numbers on its vertices in the aforementioned figure. Let $H$ be the graph from \cref{fig:nadara} (right) and observe that $H$ is a topological minor of $G$ that can be obtained by contracting the edge $kc$.
We show that $G$ and $H$ satisfy the statement of the proposition.
First, we verify that $\psi$
is a linear coloring by checking all the paths in $G$. 
Towards a contradiction suppose there is a path $P$ that does not have a center. 
Then, $P$ must see one of the colors 1 and 2.
Suppose it sees 1 but not 2.
Then, in order to see color 1 twice, $P$ contains vertices $g$, $f$, and $a$ in order.
It follows that $P$ contains a center of color 0 or 3, a contradiction.
By symmetry, it remains to analyze the case when $P$ sees both 1 and 2.
In order to see colors 1 and 2 twice, $P$ starts in $g$ and ends in $h$. It cannot use $e$
 (otherwise $P=gfedh$, admitting $g$ as a center) so it starts with $gfa$ and ends with $cdh$. 
 In the rest of the graph (i.e., $G[\{i,b,j,k\}]$) no component contains twice the color $0$, so any way to complete the path contains a center of color 0, a contradiction.

\begin{figure}[ht]
    \centering
    \includegraphics{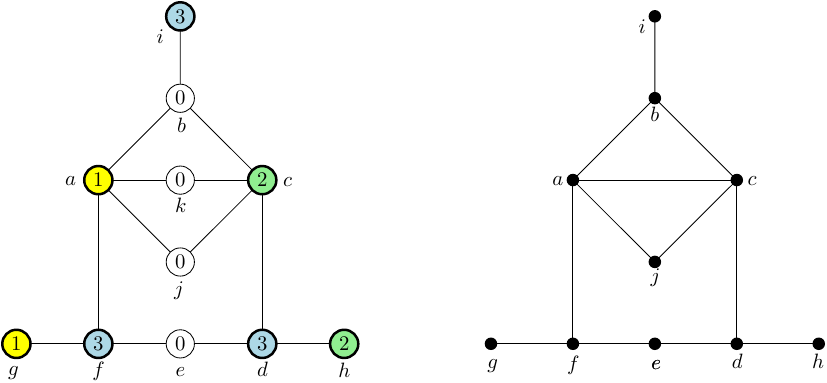}
    \caption{A graph $G$ with $\clin(G)=4$ (left) and a topological minor $H$ of $G$ (right) with $\clin(H)>4$.}
    \label{fig:nadara}
\end{figure}

Note that three colors are not enough for a linear coloring as $G$
contains a path on 9 vertices (for instance, $\mathit{gfedcjabi}$), which requires $4$ colors, as we will note in \Cref{lem:tdpk}.
We conclude that $\clin(G)=4$.

Now, we show that $\clin(H)>4$. Suppose, for contradiction,
that there exists a linear coloring $\varphi$ of $H$ using four colors,
say $\{0,1,2,3\}$. Observe that any cycle in $H$ has at least two centers
in any linear coloring (since otherwise removing a unique center would result in a path without a center), in particular let $1,0,2$ be colors of $a,b$,
and $c$ in $\varphi$, respectively. We split the analysis according to pairs of unique
colors in the cycle $C$ with vertices $a,b,c,d,e,f$. Observe that two centers
cannot be adjacent in $C$, as otherwise the rest of $C$ is a four-vertex
path, which needs at least three colors in any linear coloring. Therefore,
if 0 is a unique color in $C$, neither 1 nor 2 is unique in $C$.
Moreover, observe that 3 also is not a unique color in $C$ as otherwise
the vertices from $C$ colored by 1 and 2 would create a four-vertex path
in $H$. This is a contradiction, since the cycle has two unique colors. 
It follows that 0 is not a unique color in $C$.

Assume for
now that colors 1 and 2 are unique in $C$. It follows that $\varphi(f)=\varphi(d)=3$
and $\varphi(e)=0$ or vice versa. In both cases, $\varphi(g),\varphi(h)\in\{1,2\}$
as any of them creates a four-vertex path with vertices $f,e,d$.
Consider paths $\mathit{bcdefg}$ and $\mathit{bafedh}$ in the case where $\varphi(e)=0$
and paths $\mathit{bcdh}$ and $\mathit{bafg}$ in the case where $\varphi(e)=3$.
Together with the previous argument, these paths force $g$ and
$h$ to have different colors in $\{1,2\}$. 
However, the path $\mathit{gfacdh}$ has no center. 
It follows that at least one of the colors 1 and 2 is not unique in $C$.

Finally, assume that 1 and 3 are unique colors in
$C$. As 2, and also 0, are not unique in $C$ by the previous analysis,
we have that the vertices $f,e$, and $d$ have different colors in
$\{0,2,3\}$. We now determine colors of $i$ and $j$. As $\varphi(i)\neq\varphi(b)$
and path $\mathit{fedcbi}$ enforces $\varphi(i)\neq3$, we have $\varphi(i)\in\{1,2\}$.
Hence, $\varphi(j)=3$ by considering path $\mathit{ibaj}$ or $\mathit{ibcj}$ and triangle $\mathit{acj}$. Moreover, $\varphi(i)=2$, by path $\mathit{jcdefabi}$.
As $\varphi(d)\neq\varphi(c)$ and $\varphi(d)\neq0$ by path $\mathit{ibcd}$, $\varphi(d)=3$.
It follows $\varphi(e)=0$ by path $\mathit{jcde}$, and then $\varphi(f)=2$.
Consider the color of vertex $g$: $\varphi(g)\neq0$ by path $\mathit{jcdefg}$,
$\varphi(g)\neq1$ by path $\mathit{dcjafg}$, $\varphi(g)\neq\varphi(f)=2$, and
$\varphi(g)\neq3$ by path $\mathit{bcdefg}$.
This contradicts that $\varphi$ is a linear coloring of $H$ with unique colors 1 and 3 in $C$, but also that $\varphi$ is a linear coloring of $H$ as we
consider all cases up to the symmetry of $1$ and $2$.
This shows that \hbox{$\clin(H)>4$}.
\end{proof}

On paths, the two invariants coincide (because every connected subgraph of a path is a path) and their value is logarithmic in their length (see \cite[Section~6.2]{nevsetvril2012sparsity}).
\begin{lemma}\label{lem:tdpk}
    For every integer $n$, $\ccen(P_n) = \clin(P_n) = \lceil \log(n+1) \rceil$.
\end{lemma}

\Cref{th:treedeg} and \Cref{lem:tdpk} imply the following.

\begin{corollary}\label{cor:treedeg}
Let $T$ be a subcubic tree.
Then $\ccen(T) \leq (\log 3)\cdot \clin(T)$.
\end{corollary}

It is known that the presence in a graph of a large complete binary tree as a minor forces a large centered chromatic number, because of \Cref{{prop:minorclosed}} and the following well-known equality.

\begin{lemma}[Folklore]\label{lem:tdbk}
For every $k\in \N$, $\ccen(B_k) = k$.
\end{lemma}
\begin{proof}
    It is easy to check that coloring each vertex of $B_k$ with its distance to the root is a centered coloring, and it has $k$ colors, so $\ccen(B_k)\leq k$. Conversely, by contradiction suppose that $\ccen(B_k)\geq k$ fails for some value of $k>1$ and let us consider the minimum such value. So $B_k$ admits a centered coloring $\varphi$ with less than $k$ colors. From the definition there is a vertex $v$ to which $\varphi$ assigns a unique color. 
    Observe that $B_k-v$ contains a copy of $B_{k-1}$, that we call $B$. 
    As $\varphi(v)$ is unique, $\varphi$ assigns at most $k-2$ colors to $B$. By minimality of $k$, $B$ has a subgraph with no unique color by $\varphi$, but this is also a subgraph of $B_k$ with this property, a contradiction.
    \end{proof}

A similar statement holds for linear chromatic number, as a consequence of \Cref{cor:treedeg}.
\begin{lemma}
       If $G$ contains $B_k$ as a minor, then $\clin(G) \geq k/\log 3$.
\end{lemma}
\begin{proof}
   Since the complete binary tree $B_k$ is subcubic, if $G$ contains $B_k$ as a minor, then it contains a subdivision $B$ of $B_k$ as a subgraph (see \cite[Proposition 1.7.3]{MR4874150}). 
    As subdividing edges never decreases $\ccen$ (\Cref{prop:minorclosed}), and by \Cref{lem:tdbk}, we infer that $\ccen(B)\geq k$.
    By \Cref{cor:treedeg}, $\clin(B)\geq k/ \log 3$.
    By \Cref{obs:monolin}, we get the desired bound.
\end{proof}

We note that the constant $\log 3$ in the statement above is tight because of the following result. This suggests a possible improvement of the ratio in \Cref{conj:main} in the special case of trees.

\begin{theorem}[\cite{kun2021polynomial}]
\label{th:log3}
$\lim\limits_{k\rightarrow \infty} \ccen(B_k) / \clin(B_k)\geq \log 3$.
\end{theorem}

\section{General bounds}
\label{sec:general}

In this section we provide improvements over the bound of \Cref{th:polybd} in minor-closed graph classes, graphs of bounded treewidth, trees, and certain intersection graphs, including chordal graphs and circular-arc graphs.

Following~\cite{kun2021polynomial}, Bose et al.~\cite{bose2022linear} relaxed the notion of a $k\times k$ grid to a family of graphs called $k\times k$ \emph{pseudogrids}, that have the following properties:
\begin{enumerate}[(1)]
\item A graph $G$ contains $\grid_k$ as a minor if and only if $G$ contains a $k\times k$ pseudogrid as a subgraph (that is, $k\times k$ cor:treedeg are exactly the subgraph-minimal graphs containing $\grid_k$ as a minor).
\item There is a constant $c>0$ such that the linear chromatic number of any $k\times k$ pseudogrid is at least $ck$.
\end{enumerate}

These two facts imply the following.

\begin{theorem}[\cite{bose2022linear}]
\label{th:lineargrid}
There is a constant $c>0$ such that if $G$ contains $\grid_k$ as a minor then $\clin(G)\geq c k$.
\end{theorem}

Czerwinski, Nadara, and Pilipczuk proved in~\cite[Theorem 1.4]{czerwinski2021improved} that large centered chromatic number implies either large treewidth or the presence of a subcubic tree with large centered chromatic number.
The precise statement of the result involves a numerical constant, denoted by $\alpha$ and defined as \hbox{$\alpha = \log 3/\log \left (\left(1+\sqrt{5}\right)/2\right)$}. (Note that $\alpha \approx 2.283$.)

\begin{sloppypar}
\begin{theorem}[\cite{czerwinski2021improved}]
\label{th:wd}
    For every two integers $w,d>0$ and every graph $G$, if ${\ccen(G) \geq \alpha wd}$, then either $\tw(G)\geq w$ or $G$ contains as a subgraph a subcubic tree $T$ with ${\ccen(T) \geq d}$.
\end{theorem}
\end{sloppypar}

We will also use the following \textsl{linear grid minor theorem} of Demaine and Hajiaghayi.

\begin{theorem}[\cite{demaine2008linearity}]
\label{th:bidim}
    For every proper minor-closed graph class $\mathcal{G}$ there is a constant $c$ such that for every $t\in \N$ and every $G\in \mathcal{G}$, if $\tw(G)\geq c\cdot t$ then $G$ contains $\grid_t$ as a minor.
\end{theorem}

We are now ready to prove the first result of this section, which is a quadratic relation between the centered and the linear chromatic numbers for graphs excluding any fixed graph as a minor.

\begin{theorem}\label{th:minclo}
    For every proper minor-closed graph class $\mathcal{G}$ there is a constant $c>0$ such that for every $G\in \mathcal{G}$,
    $\ccen(G)\leq c\cdot \clin(G)^2$.
\end{theorem}

\begin{proof}
    Let $c_{\ref{th:lineargrid}}$ and $c_{\ref{th:bidim}}$ be the constants from \Cref{th:lineargrid,th:bidim,}, respectively.
    Let $G\in\mathcal{G}$ and let $k=\ccen(G)$.
    We deal with small values separately. 
    Let $c_0 = 4\alpha (c_{\ref{th:bidim}}+1)^2$.
    If $k \leq c_0$, then the statement trivially holds for $c=c_0$.
    So in the following we may thus assume $k>c_0$.

    We apply \Cref{th:wd} with $w = d = \left \lfloor \sqrt{k/\alpha}\right \rfloor$ (indeed, we have $w,d>0$ and $k\ge \alpha wd$, so the theorem applies). 
    We distinguish two cases depending on the outcome of the theorem.
    
    If $\tw(G)\geq w$, then by \Cref{th:bidim} $G$ contains a $t\times t$ grid as a minor, where $t =  \left \lfloor w/c_{\ref{th:bidim}} \right \rfloor$.
    By \Cref{th:lineargrid}, 
    \begin{align*}
     \clin(G) &\geq c_{\ref{th:lineargrid}} \cdot t\\
     &= c_{\ref{th:lineargrid}} \left \lfloor \frac{ \left \lfloor \sqrt{k/\alpha} \right \rfloor}{c_{\ref{th:bidim}}} \right \rfloor\\
     &\geq c_{\ref{th:lineargrid}} \left ( \frac{\sqrt{k/\alpha}-1}{c_{\ref{th:bidim}}}-1 \right )\\
     &= c_{\ref{th:lineargrid}}  \frac{\sqrt{k}-\sqrt{\alpha}(c_{\ref{th:bidim}}+1)}{c_{\ref{th:bidim}}\sqrt{\alpha}}\\
     &\geq \frac{c_{\ref{th:lineargrid}}}{2c_{\ref{th:bidim}}\sqrt{\alpha}}\sqrt{k} & \text{as}\ k\geq c_0.
    \end{align*}
    So in this case $k \leq c_1 \cdot \clin(G)^2$ for $c_1 = \left ( \frac{2c_{\ref{th:bidim}}\sqrt{\alpha}}{c_{\ref{th:lineargrid}}} \right )^2$.
    
    In the remaining outcome of \Cref{th:wd}, $G$ contains a subcubic tree $T$ with $\ccen(T)\geq d$ as a subgraph. We apply \Cref{cor:treedeg} and get $\clin(G)\geq \clin(T) \geq d/\log 3 = 
    \left \lfloor \sqrt{k/\alpha}\right \rfloor / \log 3$. 
    So similarly as above, $k\leq c_2 \clin(G)^2$ for some constant $c_2>0$.
    Overall the claimed statement holds for $c = \max(c_0, c_1, c_2)$.   
\end{proof}

In graph classes of bounded treewidth we can get a linear bound.

\begin{theorem}\label{th:lineartw}
Every graph $G$ satisfies 
\[\ccen(G)\leq  c\cdot (\tw(G)+1) \cdot \left(\clin(G)+\frac{1}{\log 3}\right),\]
where $c=  \alpha \log 3 < 3.7$.
\end{theorem}

\begin{proof}
    Let $k = \ccen(G)$.
    We apply \Cref{th:wd} for $w = \tw(G)+1$ and $d =\lfloor\frac{k}{\alpha w}\rfloor$.

    Treewidth does not increase when taking subgraphs, so only the second outcome of the theorem may hold. 
    Hence, $G$ contains a subcubic tree $T$ with $\ccen(T)\geq d$.

     Since the linear chromatic number is monotone under subgraphs (see \cref{obs:monolin}), we have $\clin(G) \geq \clin(T)$.
    As $T$ is subcubic, we may use \Cref{cor:treedeg} 
    and obtain, using $d>\frac{k}{\alpha w}-1$,
    \[\clin(G) \geq \clin(T)
    \geq \ccen(T) / \log 3
    > \frac{k}{\alpha w(\log 3)}-\frac{1}{\log 3}\,.\]
    Hence, $k \leq \alpha (\log 3) (\tw(G)+1) \cdot \clin(G) +\alpha (\tw(G)+1)$, proving the claimed result.
\end{proof}

As a corollary (and using that $\ccen(G)=\clin(G)$ when $\clin(G)\leq 2$, see \cref{obs:small-values}) we get the following bound.

\begin{corollary}\label{cor:lineartw}
Every graph $G$ satisfies \[\ccen(G)\leq  c\cdot (\tw(G)+1) \cdot \clin(G)\,,\]
where $c=  \alpha(\log 3 + 1/3) < 4.38$.
\end{corollary}

In graph classes where the treewidth is linearly bounded from above by the clique number (which is a lower bound for $\clin$), \cref{th:lineartw} implies a quadratic bound. 
In particular, this is the case for intersection graphs of connected subgraphs of graphs of bounded treewidth.
For a class $\mathcal{G}$ of graphs, let $I(\mathcal{G})$ be the class of \emph{region intersection graphs} of graphs in $\mathcal{G}$, defined as follows.
For $H\in \mathcal{G}$ and a family $\{H_j\}_{j\in J}$ of connected subgraphs of $H$, let $G$ be the graph with vertex set $J$ in which two distinct vertices $i$ and $j$ are adjacent if and only if $H_i$ and $H_j$ have a vertex in common.
Region intersection graphs have been studied as a common generalization of many classes of geometric intersection graphs (see~\cite{DBLP:conf/innovations/Lee17}).
For any graph class $\mathcal{G}$ of graphs with bounded treewidth, region intersection graphs of graphs in $\mathcal{G}$ have treewidth linearly bounded from above by the clique number.
This follows from the following property, shown in the proof of~\cite[Lemma 2.4]{MR1642971}.

\begin{lemma}[\cite{MR1642971}]\label{lem:BGT}
Let $k\ge 0$ be an integer and let $\mathcal{G}$ be a class of graphs with treewidth at most $k$.
Then each graph $G\in I(\mathcal{G})$ has a tree decomposition in which each bag is a union of at most $k+1$ cliques.
\end{lemma}

\begin{theorem}\label{thm:rigs}
Let $k\ge 0$ be an integer and let $\mathcal{G}$ be a class of graphs with treewidth at most $k$.
Then each graph $G\in I(\mathcal{G})$ satisfies $\ccen(G) \leq c (k+1)(\clin(G)^2 + \clin(G)/\log 3)$, where $c=  \alpha \log 3 < 3.7$.
\end{theorem}

\begin{proof}
Let $G\in I(\mathcal{G})$. 
By \Cref{lem:BGT}, $G$ has a tree decomposition in which each bag is a union of at most $k+1$ cliques.
Hence, $\tw(G)\le (k+1)\omega(G)-1$.
By \cref{th:lineartw} and using the fact that $\omega(G)\le \clin(G)$ we infer the following.

\begin{align*}
  \ccen(G) &\leq c (\tw(G) +1)(\clin(G)+ 1/\log 3)\\
  &\leq c (k+1)(\clin(G)^2 + \clin(G)/\log 3)\,.\qedhere
\end{align*}%
\end{proof}

\begin{corollary} \label{cor:rigs} 
For every class $\mathcal{G}$ of graphs with bounded treewidth, each graph $G\in I(\mathcal{G})$ satisfies $\ccen(G) = \Oh(\clin(G)^2)$.
\end{corollary}

Since every chordal graph is an intersection graph of subtrees in a tree (see~\cite{Buneman1974,Walter1978,Gavril1974}) and every circular-arc graph is an intersection graph of paths in a cycle, the cases $k \in \{1,2\}$ of \cref{cor:rigs} imply the following generalizations of the analogous result from \cite{kun2021polynomial} for interval graphs (\Cref{th:intv}).

\begin{corollary}\label{cor:chordal}
Every chordal or circular-arc graph $G$ satisfies $\ccen(G) = \Oh(\clin(G)^2)$.
\end{corollary}

A crucial step in the proof of \Cref{th:wd} is the following.
\begin{lemma}[\cite{czerwinski2021improved}]\label{lem:sparsificatree}
    Every tree $T$ contains a subcubic tree $T'$ with 
    $\ccen(T') \geq \ccen(T) / \alpha$ as a subgraph.
\end{lemma}

Using \Cref{lem:sparsificatree} we can improve by a factor of $2$ the ratio in the bound of \Cref{th:lineartw} in the case of trees, with a similar proof.

\begin{sloppypar}
\begin{theorem}\label{thm:trees}
For every tree $T$, $\ccen(T)\leq c\cdot \clin(T)$, where
$c=  \alpha \log 3 < 3.7\,.$
\end{theorem}
\end{sloppypar}

\begin{proof}
  By \Cref{lem:sparsificatree}, $T$ contains a subcubic tree $T'$ with  $\ccen(T') \geq \ccen(T) / \alpha$ as a subgraph.
    Using \Cref{cor:treedeg} we obtain
    $\clin(T) \geq \clin(T')
    \geq \ccen(T') / \log 3
    \geq \ccen(T)/(\alpha (\log 3))$.
\end{proof}

\section{When centered is linear or almost linear}\label{sec:lin=cen}

Interesting cases of graphs satisfying \cref{conj:main} are the graphs for which the two chromatic numbers are equal.
A sufficient condition for this is that every linear coloring of the graph is centered. 
Three graph classes with this property were already identified by Kun et al.~\cite{kun2021polynomial}.
In \Cref{subsec:lin=cen}, we identify two more.

In \Cref{sec:exact}, we consider two more graph families for which the linear and centered chromatic numbers coincide and determine their exact values: complete multipartite graphs and complements of rook's graphs.

Finally, in \Cref{sec:plusone}, we show that for caterpillars the two parameters differ by at most one.

We conclude this preamble with definitions and a simple lemma that are used in several of the following subsections.

Given a positive integer $k$, a \emph{complete $k$-partite graph} is any graph $G$ such that $V(G)$ admits a partition into $k$ nonempty parts such that two vertices of $G$ are adjacent if and only if they belong to different parts.
A \emph{complete multipartite graph} is any graph $G$ that is complete 
\hbox{$k$-partite} for some $k$.
Note that if $G$ is complete multipartite, then $G$ is complete $k$-partite for a unique $k$; moreover, the corresponding partition into $k$ \emph{parts} is unique.
Complete multipartite graphs are cographs, hence, every linear coloring is centered (see~\cite{kun2021polynomial}).

\begin{lemma}\label{lem:join}
Let $G$ be a graph, let $(A,B)$ be a bipartition of $V(G)$ such that every vertex of $A$ is adjacent to every vertex of $B$, and let $\varphi$ be a linear coloring of $G$.
Then, in one of $A$ and $B$, the coloring $\varphi$ assigns different colors to all the vertices.
\end{lemma}

\begin{proof}
Suppose for a contradiction that two vertices $a,a'\in A$ are assigned the same color and two vertices $b,b'\in B$ are assigned the same color.
Then these four vertices form a $P_4$-subgraph with no center. 
\end{proof}

\subsection{Graphs for which every linear coloring is centered}\label{subsec:lin=cen}

Recall that every centered coloring of a graph $G$ is also a linear coloring.
Hence, a sufficient condition for the equality between the linear and centered chromatic numbers of $G$ is that every linear coloring of $G$ is centered.

Kun et al.~\cite{kun2021polynomial} observed that this property holds for graphs with independence number at most two, cographs, and graphs with maximum degree at most $2$.
The result for graphs with independence number at most two can be generalized as follows.

\begin{theorem}\label{th:p3p1}
    If $G$ is a $(P_3+P_1)$-free graph, then every linear coloring of $G$ is centered.
     In particular, $\clin(G) = \ccen(G)$.
\end{theorem}

\begin{proof}
   Suppose for a contradiction that $G$ has a non-centered linear coloring $\varphi$ and let $H$ be a connected induced subgraph of $G$ with no center.
   We consider two cases depending on whether the complement $\overline{H}$ of $H$ is connected or not.
      \paragraph{Case 1.} $\overline{H}$ is connected. 
        Since $\overline{H}$ is paw-free, by a result of Olariu \cite{MR0947254} $\overline{H}$ is either triangle-free or complete multipartite. 
        In the first case the independence number of $H$ is at most $2$, so by \cite[Lemma 1]{kun2021polynomial} $\varphi$ is a centered coloring, a contradiction. 
        In the second case, as $H$ is connected, $H$ is a complete graph and the same conclusion follows.  
    \paragraph{Case 2.} $\overline{H}$ is disconnected. 
        Then $H$ has a bipartition $(A,B)$ of its vertex set into two nonempty sets 
        with all edges between them. 
        As $\varphi$ is a proper coloring, $\varphi$ colors $A$ and $B$ with disjoint sets of colors. 
        In addition, by \Cref{lem:join}, in one of $A$ and $B$ $\varphi$ assigns different colors to all the vertices.
        We may assume without loss of generality that the vertices of $A$ get colored with different colors. 
       Then, any vertex of $A$ is a center of $H$, contradicting the choice of $H$.
\qedhere
\end{proof}

A useful observation, also used in \cite{kun2021polynomial}, is that the property holds for graphs such that every connected induced subgraph has a Hamiltonian path.
\begin{observation}\label{obs: Hamiltonian}
    If $G$ is such that every connected induced subgraph of $G$ has a Hamiltonian path, then every linear coloring of $G$ is centered.
\end{observation}
Indeed, if $H$ admits a linear coloring $\psi$, then for every connected subgraph $H$ of $G$, $V(H)$ induces a connected subgraph that has a Hamiltonian path $P_H$. Then $P_H$ admits a center, thus $H$ admits a center, and therefore $\psi$ is a centered coloring of $G$.

A graph $G$ is an \emph{interval graph} if it is an intersection graph of a family of closed intervals on the real line. 
If, moreover, none of the intervals contains another, then $G$ is a \emph{proper interval graph}.
Kun et al.~\cite{kun2021polynomial} showed that every interval graph $G$ satisfies $\ccen(G) \le \clin(G)^2$.
If $G$ is a proper interval graph, then $\clin(G) = \ccen(G)$, as implied by the fact that every proper interval graph is (claw, net)-free\footnote{The \emph{claw} is the graph obtained by attaching 3 pendant vertices to a vertex, and the \emph{net} is the graph obtained by attaching a pendant vertex to each vertex of a $K_3$.} and the following.
\begin{proposition}\label{th:CN-free}
If $G$ is (claw, net)-free, then every linear coloring of $G$ is centered. In particular, $\clin(G) = \ccen(G)$.
\end{proposition}

\begin{proof}
Duffus, Jacobson, and Gould~\cite{MR0634535} showed that every connected (claw, net)-free graph has a Hamiltonian path. Moreover, the class of (claw, net)-free graphs is closed under induced subgraphs. 
The result is thus implied by \cref{obs: Hamiltonian}.
\end{proof}

\subsection{Complete multipartite graphs and complements of rook's graphs}\label{sec:exact}

\begin{theorem}\label{compbip}
Let $G$ be a complete multipartite graph. 
Then $\clin(G) = \ccen(G)$, with both quantities equal to $n-p+1$ where $n$ is the order of $G$ and $p$ is the maximum cardinality of a part of~$G$.
\end{theorem}

\begin{proof}
To see that $\ccen(G) \le n-p+1$, observe that we can color the largest part $Y$ of $G$ with the same color and all the other vertices with different colors. This is a centered coloring since every subgraph of $G$ either contains a single vertex from $Y$, or contains a uniquely colored vertex from $G-Y$.

To show that $\clin(G) \ge n-p+1$, consider a linear coloring $\varphi$ and let $k$ be the number of colors used.
If all vertices of $G$ are colored with unique colors, then $k = n\ge n-p+1$.
Hence, we may assume that two vertices $u,v$ receive the same color.
Then $u$ and $v$ are non-adjacent and thus belong to the same part $X$ of $G$.
Since the pair $(X, V(G)\setminus X)$ is a bipartition of $V(G)$ with all edges between $X$ and $V(G)\setminus X$, by \Cref{lem:join}, $\varphi$ assigns different colors to all the vertices in one of $X$ and $V(G)\setminus X$.
Since $\varphi(u) = \varphi(v)$, we infer that $\varphi$ assigns different colors to all the vertices in $V(G)\setminus X$.
Furthermore, since $\varphi$ is a proper coloring, no color of a vertex in $V(G)\setminus X$ can be used on any vertex in $X$.
Hence, $k\ge n-|X|+1\ge n-p+1$, as claimed.
We conclude that $\clin(G) \le n-p+1\le \ccen(G)\le \clin(G)$, therefore, equalities must hold throughout.
\end{proof}

For all $n,m\geq 1$, the $n\times m$ \emph{rook's graph} is a graph on $nm$ vertices, arranged into $m$ columns with $n$ vertices each and into $n$ rows with $m$ vertices each, such that two vertices are adjacent if and only if they are in the same column or in the same row.

We consider here the \emph{co-rook's graph}, that is, the complement of the $n\times m$ rook's graph.
Note that this graph can also be obtained by removing the edges of $n$ disjoint complete graphs $K_m$ from the complete multipartite graph with $m$ parts with $n$ vertices each.

\begin{theorem}\label{corook}
    Let $n\geq m \geq 1$ and let $G_{n,m}$ be the complement of the $n\times m$ rook's graph.
    Then 
    \[\ccen(G_{n,m})=\clin(G_{n,m}) = 
    \begin{cases}
        nm - n +1 &\text{if \,$m\ge 2$ and $n\ge 3$\,,}\\
        m &\text{otherwise\,.}
    \end{cases}\]
\end{theorem}

\begin{proof}
Suppose first that $m=1$. 
Then $G_{n,1}$ is an edgeless graph of order $n$, thus $\ccen(G_{n,1})=\clin(G_{n,1})=1$.
Suppose next that $m=n=2$. 
Then $G_{2,2}$ is a disjoint union of two complete graphs $K_2$, hence $\ccen(G_{2,2})=\clin(G_{2,2})=2$.

Assume now that $m\geq 2$ and $n\geq 3$, with $n\geq m$.
Since $G_{n,m}$ is a subgraph of the complete multipartite graph with $m$ parts of $n$ vertices, we obtain from \Cref{prop:minorclosed,compbip} that 
$\ccen(G_{n,m})\le nm-n+1$; so let us show that $\clin(G_{n,m})\geq nm - n +1$.
Suppose for a contradiction that there exists a linear coloring $\varphi$ of $G_{n,m}$ with at most $nm-n$ colors.

Recall that the vertices of $G_{n,m}$ are partitioned into $m$ columns $C_1,\dots, C_m$ and into $n$ rows $R_1,\dots, R_n$. 
Two vertices are non-adjacent if and only if they belong to the same column or row.
Hence, for any two vertices $u$ and $v$, if $\varphi(u)=\varphi(v)$, then they are non-adjacent, so either there is $i\in \{1,\dots,n\}$ such that $u,v\in R_{i}$ or there is $j\in \{1,\dots,m\}$ such that $u,v\in C_{j}$.

\begin{claim}\label[claim]{cl:twoCouples}
    Let $u,v,u',v'$ be four distinct vertices 
    such that $\varphi(u)=\varphi(v)$ and $\varphi(u')=\varphi(v')$. 
Then either three of those vertices belong to the same column or to the same row, or there exist two rows $R$ and $R'$ and two columns $C$ and $C'$ such that $\{u,v,u',v'\}\subseteq R\cup R'$ as well as $\{u,v,u',v'\}\subseteq C\cup C'$.
\end{claim}

\begin{proof}
    Recall that the vertices in each pair are either in the same column or in the same row.
    
    Suppose that $u$ and $v$ are in the same column $C_j$, and respectively in rows $R_{i_u}$ and $R_{i_v}$.
    Suppose that neither $u'$ nor $v'$ is in $C_j$ (otherwise there are three of those vertices in $C_j$), and that one of them, without loss of generality $u'$, is neither in $R_{i_u}$ nor in $R_{i_v}$. 
    Thus, $u'$ is adjacent to both $u$ and $v$. 
    Moreover, since $v'$ cannot be both in $R_{i_u}$ and $R_{i_v}$, $v'$ is adjacent to at least one of $u,v$, without loss of generality $v$. 
    Then $uu'vv'$ is a path of $G_{n,m}$ with no center, a contradiction.

    Therefore, $u',v'$ belong to the union of the rows $R_{i_u}$ and $R_{i_v}$.
    If they are in the same row, then three of the vertices $u,v,u',v'$ are in a same row, so we may assume (w.l.o.g.) that $u'$ belongs to $R_{i_u}$ and $v'$ belongs to $R_{i_v}$.
    Suppose that $u',v'$ are not both in the same column: then $uv'u'v$ is a path of $G_{n,m}$ with no center, a contradiction.

    The case where $u,v$ are in a same row is obtained by symmetry.
    \cqed
\end{proof}

\begin{claim}\label[claim]{cl:oneCouplebyPart}
    For each column (resp.~row), there are at most two vertices whose color is not unique in this column (resp.~row).
\end{claim}
\begin{proof}
    Suppose that there exist two distinct vertices $u, v\in V(G_{n,m})$ that are in the same column $C_j$ such that $\varphi(u)=\varphi(v)$.
    Let $R_{i_u}$ and $R_{i_v}$ be the two rows containing $u$ and $v$, respectively.

    Observe that the graph $G'= G_{n,m} - C_j$ contains exactly $nm-n$ vertices, and none of these vertices can be colored $\varphi(u)$, so if each vertex of $G'$ has a unique color in $G'$, then $\varphi$ uses at least $|V(G')|+1>nm-n$ colors, a contradiction.
    Thus there exist two distinct vertices $x$ and $y$ in $G_{n,m}-C_j$ such that $\varphi(x)=\varphi(y)$.
    By \cref{cl:twoCouples}, they are in rows $R_{i_u},R_{i_v}$, either both in one of them (then without loss of generality $R_{i_u}$) or one in each row, and both in the same column.

    Suppose now that there is another pair $u'$ and $v'$ of distinct vertices in $C_j$ such that $\varphi(u')=\varphi(v')$, $u'\notin \{u,v\}$, and $v'\neq u$ (possibly $v'=v$ and $\varphi(u)=\varphi(v)=\varphi(u'))$.
    Then if $x,y$ are both in $R_{i_u}$, they are not in the same row as $u'$ nor $v'$, and if they are in the same column, none of them is in the same row as $u'$. 
    In both cases, this contradicts \cref{cl:twoCouples}.

    Finally, the case when $u$ and $v$ are in the same row $R_i$ is obtained by symmetry and observing that $|V(G_{n,m} - R_i)|=nm-m \geq nm- n $.
    \cqed
\end{proof}

\begin{claim}\label[claim]{cl:atMostTwoColors}
Each color appears at most twice in $G_{n,m}$.
\end{claim}

\begin{proof}
Suppose that $G_{n,m}$ contains three distinct vertices $u$, $v$, and $w$ such that $\varphi(u) = \varphi(v) = \varphi(w)$.
Since these three vertices are pairwise non-adjacent, they either all belong to the same row or to the same column.
This contradicts \Cref{cl:oneCouplebyPart}.
\end{proof}

\begin{claim}\label[claim]{cl:3couples}
Let $\{u_k,v_k\}$, $k\in\{a,b,c\}$, be three disjoint pairs of vertices such that $\varphi(u_k)=\varphi(v_k)$ for all $k\in\{a,b,c\}$.
Then, four vertices among those six belong to the same column or to the same row, and all the other $nm-6$ vertices have a unique color.
\end{claim}

\begin{proof}
Every pair of vertices of the same color has to be either in the same column or in the same row.
We call such a pair a \emph{column} pair (resp.~a \emph{row} pair).
Hence, we may assume without loss of generality that $\{u_a,v_a\}$ and $\{u_b,v_b\}$ are both column pairs or both row pairs.
Suppose that they are both column pairs (the case when they are both row pairs is treated symmetrically).
Note that since $\varphi(u_a)=\varphi(v_a)$ and $\varphi(u_b)=\varphi(v_b)$, we infer from \Cref{cl:oneCouplebyPart} that the columns containing $\{u_a,v_a\}$ and $\{u_b,v_b\}$, respectively, are distinct.
Up to reordering the columns and rows, we can assume that $\{u_a,v_a\}$ and $\{u_b,v_b\}$ are respectively in $C_1$ and $C_2$.
By \cref{cl:twoCouples}, we can also assume that $u_a,u_b$ belong to $R_1$ and  $v_a,v_b$ belong to $R_2$.

Suppose that $\{u_c,v_c\}$ is also a column pair.
Let $C_j$ be the column containing $\{u_c,v_c\}$.
By \cref{cl:oneCouplebyPart}, $j\notin \{1,2\}$, and by \cref{cl:twoCouples} applied to $\{u_a,v_a\}$ and $\{u_c,v_c\}$, we can assume that $u_c\in R_1$ and $v_c\in R_2$. 
Therefore, $u_av_bu_cv_au_bv_c$ is a path in $G_{n,m}$ with no center, a contradiction with the fact that $\varphi$ is a linear coloring of $G_{n,m}$.

Thus, $\{u_c,v_c\}$ has to be a row pair, and let $R_i$ be the row containing $\{u_c,v_c\}$.
Suppose first that $i\notin \{1,2\}$.
Then by \cref{cl:twoCouples} applied to $\{u_a,v_a\}$ and $\{u_c,v_c\}$, we can assume that $u_c\in C_1$ and $v_c\in C_2$.
Then, $u_av_bu_cu_bv_av_c$ is a path in $G_{n,m}$ with no center, a contradiction.
Therefore, $\{u_c,v_c\}$ are either in $R_1$ or $R_2$, say $R_1$, hence there are four vertices (namely, $u_a$, $u_b$, $u_c$, and~$v_c$) among those six that are on the  same row.

It remains to show that all the other $nm-6$ vertices have a unique color.
    Suppose for a contradiction that there exists a fourth pair $\{u_d,v_d\}$ such that $\varphi(u_d)=\varphi(v_d)$.
    Applying the previous observations for $\{u_c,v_c\}$ to $\{u_d,v_d\}$, we infer that $\{u_d,v_d\}$ are either in $R_1$ or $R_2$. 
    Since $\{u_c,v_c\}$ are in $R_1$, we infer from~\cref{cl:oneCouplebyPart} that $\{u_d,v_d\}$ are in $R_2$.
    By \cref{cl:twoCouples}, we may assume without loss of generality that $u_c,u_d$ and $v_c,v_d$ are respectively in the same column.
    Then $u_au_dv_cv_au_cv_d$ is a path in $G_{n,m}$ without a center, a contradiction.
    \cqed
\end{proof}

Recall that $\varphi$ uses at most $nm-n$ colors.
By \cref{cl:atMostTwoColors}, each of the colors appears at most twice.
Let $k$ be the number of colors that appear twice.
Then, the total number of colors used by $\varphi$ is $nm-k$, and, hence, $k\ge n$.
On the other hand, $k\le 3$, since $k\ge 4$ would contradict \cref{cl:3couples}. 
Therefore, $n\le k\le 3$, and since $n\ge 3$, we infer that $m\le n = k = 3$. 
Since $k = 3$, there exist three disjoint pairs of vertices, each of the same color.
Hence, by \cref{cl:3couples}, four vertices among those six belong to the same column or to the same row.
However, this implies that $m\ge 4$ or $n\ge 4$, respectively, a contradiction.
\end{proof}

\begin{remark}\label{rem:corook}
    Not every linear coloring of a co-rook graph is centered. Consider the graph depicted in \Cref{fig:rem-coroook}, with the coloring given by the numbers.
    
    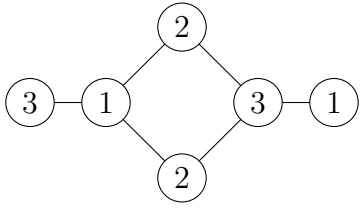
\begin{figure}[H]
        \centering
    \begin{tikzpicture}[every node/.style = {draw, circle, inner sep=3pt, minimum size=0, fill=white}]
        \draw (-1,0) node (v1) {1} -- (0,1) node (v2) {2} -- (1,0) node (v3) {3} -- (0, -1) node (v4) {2} -- cycle;
        \draw (v1) --++(-1,0) node {3};
        \draw (v3) --++(1,0) node {1};
    \end{tikzpicture}
        \caption{The coloring of \Cref{rem:corook}}
        \label{fig:rem-coroook}
    \end{figure}
    It is easy to check that this is a linear coloring; however, it is not a centered coloring, as there is no unique color. This graph is an induced subgraph of the complement $G_{2,4}$ of the $2\times 4$-rook's graph. 
    So any extension of the above coloring into a linear coloring of $G_{2,4}$ (for instance, obtained by giving unique colors to the remaining vertices) is not a centered coloring.
\end{remark}

\subsection{Linearly coloring caterpillars}
\label{sec:plusone}

While \Cref{conj:main} is still open for trees, \Cref{thm:trees} establishes a relaxed version of it.
We now provide further support for \Cref{conj:main} for trees, by proving it for caterpillars.
A \emph{caterpillar} is a tree such that the removal its of leaves yields a path, called its \emph{central path}.

\begin{theorem}
\label{th:cat}
    If $T$ is a caterpillar, then $\ccen(T)\leq \clin(T)+1$.
    Furthermore, there exist caterpillars attaining equality.
\end{theorem}
\begin{proof}
    Let $P$ be the central path of $T$, and $k = \lceil \log(|V(P)|+1) \rceil$. 
    We know from \cref{lem:tdpk} that $\ccen(P) = \clin(P) = k$. 
    We define a coloring for $T$ as follows: the path $P$ is colored according to a centered coloring achieving $\ccen(P)=k$ and all the leaves of $T$ receive a $(k+1)$st color. 
    This is a centered coloring of $T$ as each connected subgraph $T'$ of $T$ that has more than one vertex is composed of a nonempty subpath of $P$, which has a center, and a (potentially empty) subset of leaves of $T$ that do not share any color with the subpath. 
    Thus, the center of the subpath is also a center of $T'$.
    Therefore, $\ccen(T)\leq k+1$. Moreover, $T$ contains $P$ with $\clin(P)=k$ as a subgraph so by \cref{obs:monolin}, $\clin(T)\geq k$. Hence, $\ccen(T)\leq \clin(T)+1$.
\end{proof}

\begin{remark}\label{rem:op}
The bound of \Cref{th:cat} is tight, as witnessed by the graph depicted in \Cref{fig:rem-op} which has centered chromatic number 4 (see~\cite{dvorak2012forbidden}) and linear chromatic number 3. 
A linear coloring with 3 colors is given by the numbers on the vertices.
\end{remark}
\begin{figure}[H]
    \centering
    \begin{tikzpicture}[every node/.style = {draw, circle, inner sep=3pt, minimum size=0, fill=white}]
        \draw (0,0) node (v0) {1}
        \foreach \i/\j in {1/3,2/2,3/1,4/3,5/1}{
        -- (\i,0) node (v\i) {\j}
        };
        \draw (v2) -- ++(0,-1) node {3};
        \draw (v3) -- ++(0,-1) node {2};
    \end{tikzpicture}
    \caption{The coloring of \Cref{rem:op}.}
    \label{fig:rem-op}
\end{figure}
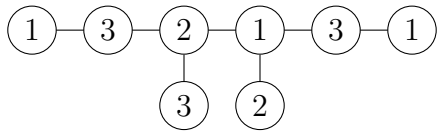

\section{Obstructions to bounded linear chromatic number}
\label{sec:obs}

A partial order is a \emph{well-quasi-order} if it contains neither an infinite decreasing sequence, nor an infinite antichain.

\begin{theorem}[\cite{nevsetvril2012sparsity}]
\label{th:bdtdwqo}
    For every $k$ the class of graphs of centered chromatic number at most $k$ is well-quasi-ordered by the induced subgraph relation.
\end{theorem}

As graphs of bounded linear chromatic number have bounded centered chromatic number (\Cref{th:polybd}), \Cref{th:bdtdwqo} implies the following.
\begin{corollary}
\label{th:bdcclinwqo}
    For every $k$ the class of graphs of linear chromatic number at most $k$ is well-quasi-ordered by the induced subgraph relation.    
\end{corollary}

The obstructions for having linear chromatic number at most $k$ have (by minimality) linear chromatic number $k+1$ and form an antichain, so we also get the following

\begin{corollary}   \label{cor:finite}
For every $k\in \N$ the class of graphs $G$ with $\clin(G)\leq k$ has a finite number of obstructions with respect to the induced subgraph relation.
\end{corollary}

Since a class of graphs closed under subgraphs has a finite obstruction set with respect to the subgraph relation if and only if it has a finite obstruction set with respect to the induced subgraph relation, \Cref{obs:monolin} implies that the analogue of \Cref{cor:finite} also holds for the subgraph relation.

\begin{corollary}   \label{cor:finite-subgraph}
For every $k\in \N$ the class of graphs $G$ with $\clin(G)\leq k$ has a finite number of obstructions with respect to the subgraph relation.
\end{corollary}

In the next section, we determine the obstructions sets for classes of graphs with $\clin(G)\leq k$ with respect to the subgraph relation, for $k\le 3$.

\subsection{Small values}



 We investigated which graphs have different values for the centered and coloring numbers, for small values.

 \begin{observation}\label{obs:small-values}
 A graph $G$ satisfies $\clin(G) \le 1$ if and only if $G$ is edgeless, and $\clin(G) \le 2$ if and only if $G$ is a star forest.
 In both cases $\clin(G) = \ccen(G)$.
 \end{observation}
 
\begin{proof}
The case $\clin(G) \le 2$ holds due to the fact that if $\clin(G) \le 2$, then $G$ excludes $P_4$ as well as all cycles as subgraphs, hence, $G$ is a star forest, implying $\ccen(G)\le 2$.
\end{proof}

\begin{sloppypar}
In 2012, Dvo\v{r}\'{a}k, Giannopoulou, and Thilikos characterized the family of forbidden subgraphs for treedepth---or, equivalently, the centered chromatic number---at most $3$ (see~\cite{dvorak2012forbidden}).
The list is finite, consisting of $14$ graphs with at most $8$ vertices.
We adapt their approach and obtain a similar result for the linear chromatic number.
The resulting list can be obtained from the list for treedepth (see~\cite[Theorem 4]{dvorak2012forbidden}) by extending two of its members by the addition of one or two pendant edges, respectively.
\end{sloppypar}

\begin{figure}[ht]
    \centering
    \includegraphics[width=\textwidth]{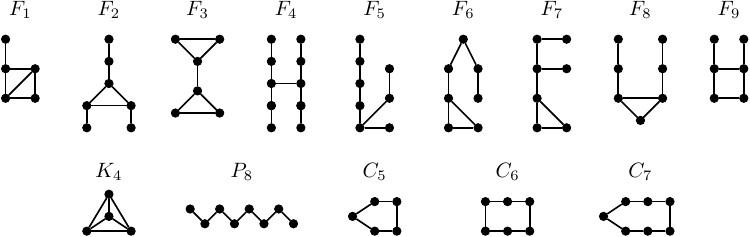}
    \caption{The list $\mathcal{F}$ of subgraph obstructions to $\clin(G)\leq 3$.}
    \label{fig:listObs}
\end{figure}

\begin{theorem}
Let $\mathcal{F}$ be the set of graphs depicted in \Cref{fig:listObs} and let $G$ be a graph.
Then, $\clin(G)\le 3$ if and only if $G$ contains none of the graphs in $\mathcal{F}$ as a subgraph.
\end{theorem}

\begin{proof}
    It can be verified that all the graphs in $\mathcal{F}$ have linear chromatic number at least $4$ (we provide a proof of this fact in \cref{appendix-small-values}). 
    Hence, by \Cref{obs:monolin}, if $G$ contains one of the graphs in $\mathcal{F}$ as a subgraph, then $\clin(G)\geq  4$. 
    
    Suppose the converse is not true, and let $G$ be a minimal counterexample, i.e., $G$ is a graph such that $G$ contains none of the graphs in $\mathcal{F}$ as a subgraph, $\clin(G)\ge 4$, and every proper subgraph of $G$ has a linear coloring with at most $3$ colors.
    In particular, $G$ is connected.
    Observe that $G$ does not contain a cycle of length greater than 4, otherwise $G$ contains either $C_5$, $C_6$, $C_7$, or $P_8$ (all of which are in $\mathcal{F}$) as a subgraph.

We organize the proof in a sequence of claims establishing a number of structural properties of $G$.

\begin{claim}\label{cl:star-cpt-on-one-side}
The graph $G$ does not admit a $4$-cycle containing two non-adjacent vertices of degree~$2$ in $G$.
\end{claim}

\begin{proof}
Towards a contradiction, denote such a cycle by $C$, and let $v_1,v_2,v_3,v_4$ be a cyclic order of the vertices of $C$ such that $v_2$ and $v_4$ have degree $2$ in $G$. 
We first show that there exists some $i\in \{1,3\}$ such that $v_i$ is not an endpoint of a path of length~$2$ in the graph $G-v_{4-i}$.
For $i\in \{1,3\}$, let $X_i$ be the set of vertices of $G$ adjacent to $v_i$ but not to $v_{4-i}$.
Then, $X_1$ and $X_3$ are disjoint (by definition) and there is no path from $X_1$ to $X_3$ in the graph $G-v_1-v_3$, since otherwise $G$ would contain a cycle of length more than $4$.
Suppose for a contradiction that there is a path of length~$2$, say $v_1x_1y_1$, in the graph $G-v_3$ and a path of length~$2$, say $v_3 x_3 y_3$, in the graph $G-v_1$.
Since vertices $v_2$ and $v_4$ have degree $2$ in $G$, they are both distinct from $x_1$ and $x_3$, as well as from $y_1$ and $y_3$.
Also, $y_1$ and $y_3$ are distinct, since there is no path from $X_1$ to $X_3$ in $G-v_1-v_3$.
Note that no vertex adjacent to both $v_1$ and $v_3$ can be adjacent to a vertex outside of $C$, otherwise we would obtain an $F_9$ subgraph. 
In particular, we have $x_i\in X_i$ for $i\in \{1,3\}$.
Then, the graph $G$ contains three edge-disjoint paths of lengths $1$, $2$, and $4$ starting at $v_1$, namely $v_1v_2$, $v_1x_1y_1$, and $v_1v_4v_3x_3y_3$, respectively. 
These paths form a subgraph of $G$ isomorphic to $F_5 \in\mathcal{F}$, a contradiction.
This shows that there exists some $i\in \{1,3\}$ such that $v_i$ is not an endpoint of a path of length~$2$ in the graph $G-v_{4-i}$.

We may assume without loss of generality that $v_3$ is the center of a star component of the graph $G-v_{1}$.
We claim that the graph $G-v_{1}$ is a star forest.
Suppose for a contradiction that this is not the case.
Suppose first that $G-v_{1}$ contains a cycle, say $C'$.
Then $C'$ does not belong to the component of $G-v_{1}$ containing $v_3$ and we conclude that the graph $G-v_3$ contains a path from $v_1$ to $V(C')$; let $P$ be a shortest such path.
Note that $P$ is of length one, since otherwise the path obtained by the concatenation of the paths $v_4v_3v_2v_1$, $P$, and a path in $C'$ of length~$2$ would form a path of length at least $7$ in $G$, a contradiction with the fact that $G$ does not contain $P_8$ as a subgraph. 
A similar argument shows that $C'$ is the $3$-cycle.
But now, $G$ contains $F_6$ as a subgraph, a contradiction.
This shows that $G-v_1$ is acyclic.
Since $G-v_1$ is not a star forest by assumption, it contains a path of length $3$ as a subgraph; let $Q=w_1w_2w_3w_4$ be such a path.
As above, $Q$ does not belong to the component of $G-v_{1}$ containing $v_3$, hence, the graph $G-v_3$ contains a path from $v_1$ to $V(Q)$; let $R$ be a shortest such path.
Similarly as above, we first observe that $R$ has length one; furthermore, since $G$ does not contain $P_8$ as a subgraph, $R$ must attach to an internal vertex of $Q$, say $w_2$.
But now, $G$ contains three edge-disjoint paths of lengths $1$, $2$, and $4$ starting at $w_2$, namely $w_2w_1$, $w_2w_3w_4$, and $w_2v_1v_2v_3v_4$, respectively, forming a subgraph of $G$ isomorphic to $F_5 \in\mathcal{F}$, a contradiction.
This shows that $G-v_1$ is a star forest.

Since $G-v_1$ is a star forest, we have $\clin(G-v_1)\le 2$ (by \Cref{obs:small-values}) and consequently $\clin(G)\le 3$ (by \Cref{obs:lininduced}), a contradiction. 
\end{proof}

\begin{claim}\label{2-conn-at-most-4}
Every $2$-connected subgraph of $G$ has at most 4 vertices.
\end{claim}

\begin{proof}
Suppose for a contradiction that $G$ contains a $2$-connected subgraph $G'$ on at least $5$ vertices. 
   Then $G'$ has to contain a cycle on $4$ vertices, say $C$, with vertices $v_1,v_2,v_3,v_4$ in cyclic order. 
   Let $v_5$ be a vertex of $V(G')\setminus V(C)$.
 Then, by $2$-connectivity, there are two paths in $G'$ connecting $v_5$ to $C$ that are vertex-disjoint except in $v_5$.
 Furthermore, each of these two paths has length one, since otherwise $G'$ would contain a cycle of length greater than $4$.
 For the same reason, these two paths connect to $C$ at opposite vertices, with loss of generality at $v_1$ and $v_3$.
 Notice that $G'$ does not contain an edge with endpoints in $\{v_2,v_4,v_5\}$, as otherwise $G'$ would contain a $5$-cycle with vertex set $V(C)\cup \{v_5\}$.
 Furthermore, none of the vertices in $\{v_2,v_4,v_5\}$ is adjacent in $G$ to a vertex outside $V(C)$, since otherwise $G$ would contain $F_9$ as a subgraph. 
 This implies that $v_2$ and $v_4$ are two non-adjacent vertices of $C$ that have degree~$2$ in $G$, contradicting \cref{cl:star-cpt-on-one-side}.
 \end{proof}
 
\begin{claim}\label{2-conn-at-most-3}
Every $2$-connected subgraph of $G$ has at most 3 vertices.
\end{claim}

\begin{proof}
Suppose for a contradiction that $G$ contains a $2$-connected subgraph with at least $4$ vertices.
By \Cref{2-conn-at-most-4}, $G$ contains a cycle on $4$ vertices, say $C$, with vertices $v_1,v_2,v_3,v_4$ in cyclic order.  
Since $G$ does not contain $K_4$ as a subgraph, we may assume without loss of generality that $v_2$ is not adjacent to $v_4$.
By \cref{cl:star-cpt-on-one-side}, vertices $v_2$ and $v_4$ cannot both have degree $2$ in $G$.
We may assume without loss of generality that $v_2$ has a neighbor outside of $C$.
If also $v_1$ has a neighbor outside of $C$, then $G$ contains either $C_5$ or $F_9$ as a subgraph.
Therefore, $v_1$ does not have any neighbors outside of $C$, and by symmetry, neither does $v_3$.
\cref{cl:star-cpt-on-one-side} implies that $v_1$ and $v_3$ are adjacent, and, hence, $G$ contains $F_1$ as a subgraph, a contradiction.
\end{proof}

\begin{claim}\label{G-has-a-cycle}
The graph $G$ contains a cycle.
\end{claim}

\begin{proof}
Suppose for a contradiction that $G$ is a forest.
Since $G$ is connected, $G$ is a tree.
The minimality of $G$ also implies that $G$ does not have any pairs of twin leaves:
indeed, if $u$ and $v$ are two leaves of $G$ with the same neighbor, then any linear $3$-coloring of $G-v$ can be extended to a linear $3$-coloring of $G$ by assigning to $v$ the color of $u$, contradicting the fact that $\clin(G)\geq 4$.
Let $P = v_1\dots v_k$ be a longest path in $G$.
Then $k\leq 7$ since $G$ contains no $P_8$ and $k\ge 4$ since otherwise $G$ is a star, implying by \Cref{obs:small-values} that $\clin(G)\le 2$, a contradiction.     
By the maximality of $P$, $d(v_1)= d(v_k) =1$. 
By the absence of twin leaves, $d(v_2) = d(v_{k-1}) = 2$.
We consider three cases depending on the value of $k$.
    \begin{enumerate}
        \item Case $k\in \{4,5\}$.
        In this case, the maximality of $P$ implies that every vertex of $G$ is at distance at most $2$ from $v_3$.
        Therefore, deleting $v_3$ from $G$ results in a star forest, implying by \Cref{obs:small-values,obs:lininduced} that $\clin(G)\le 3$, a contradiction.
        \item Case $k=6$. 
        Only $v_3$ and $v_4$ may have neighbors outside of $P$. 
        However, by the maximality of $P$, neither of $v_3$ and $v_4$ can be an endpoint of a path of length more than $2$ that is edge-disjoint from $P$.
        Moreover, $v_3$ and $v_4$ cannot both be endpoints of paths of length $2$ that are edge-disjoint from $P$ without creating an $F_4$ subgraph.
        We can thus assume without loss of generality that every path with endpoint $v_4$ that is edge-disjoint from $P$ has length at most $1$.  
        Moreover, since $G$ has no twin leaves, there can be only one such leaf $v$. 
        But now, $G$ admits a linear coloring with three colors as follows. 
        Vertex $v_5$ gets color $1$, vertices $v_4$ and $v_6$ get color $2$, vertex $v_3$ gets color $3$, and so does $v$ (if it exists).   
        Then, for the remaining vertices, we color all vertices at distance $i$ from $v_3$ with color $i$ for $i\in \{1,2\}$. 
        It is easy to check that this is indeed a linear $3$-coloring of $G$, a contradiction.
\item Case $k=7$. 
        As $G$ does not have $F_5$ as a subgraph, $d(v_3) = d(v_5) = 2$. 
        So $v_4$ is the only vertex of $P$ that can have degree more than $2$.
        Suppose that $v_4$ has a neighbor $v$ that does not belong to $P$.
        Then, by the maximality of $P$, the vertex $v$ cannot be an endpoint of a path of length more than $2$ in the graph $G-v_4$.
        Furthermore, if $v$ is an endpoint of such a path of length~$2$, then $v$ has degree~$2$ in $G$, since otherwise $G$ would contain $F_5$ as a subgraph.
        Hence, each component of $G-v_4$ is a star, implying by \Cref{obs:small-values,obs:lininduced} that $\clin(G)\le 3$,  a contradiction. 
        \qedhere
    \end{enumerate}
\end{proof}

By \Cref{2-conn-at-most-3} and \Cref{G-has-a-cycle}, every cycle in $G$ has length $3$ and, moreover, $G$ contains a $K_3$, say $C$, with vertices $v_1,v_2,v_3$. 
Observe also that every $K_3$ in $G$ shares a vertex with $C$.
Indeed, if $G$ contains another $K_3$ disjoint from $C$, then it contains either $F_3$ or $F_6$ as a subgraph. 
In particular, by \Cref{2-conn-at-most-3}, there is a vertex in $C$, without loss of generality $v_1$, that intersects all copies of $K_3$ in $G$.

\begin{claim}\label{G-has-no-K3-other-than-C}
The graph $C$ is the only copy of $K_3$ in $G$.
\end{claim}

\begin{proof}
Observe that since $G-v_1$ is not a star forest, $G-v_1$ contains a $P_4$. 
Moreover, $v_1$ is adjacent to a vertex from some $P_4$ of $G-v_1$, as otherwise we can shift a $P_4$ towards $v_1$. We denote such a path by $P$.

Suppose for a contradiction that $G$ contains a second copy $C'$ of $K_3$ (i.e., $C'\neq C$, but they may share vertices or edges). As noted above, $C'$ contains $v_1$. Also, $C$ and $C'$ cannot share other vertices because of \Cref{cl:star-cpt-on-one-side} and the fact that $K_4$ belongs to $\mathcal{F}$. 
Observe that $P$ does not intersect both $C$ and $C'$, as that would result in a longer cycle.
Thus, either $C$ or $C'$ is disjoint from $P$. 
Since $v_1$ is adjacent to a vertex from $P$, this gives rise to either $F_6$ or $F_7$, a contradiction.
\end{proof}

As in the proof of \Cref{G-has-no-K3-other-than-C},
for every vertex $v\in C$, the graph $G-v$ contains a $P_4$ adjacent to $v$.
We say that $v\in V(C)$ is \textit{good} if the unique edge in $C-v$ is the middle edge of all the copies of $P_4$ in $G-v$.

\begin{claim}
There is at most one good vertex of $C$.
\end{claim}

\begin{proof}
Suppose without loss of generality that $v_1$ and $v_2$ are good.
Then, since $v_1$ is good, the connected component of $G-v_1$ containing $v_2,v_3$ is a tree consisting of the edge $v_2v_3$ with some pendant edges.
Similarly, the connected component of $G-v_2$ containing $v_1,v_3$ is a tree consisting of the edge $v_1v_3$ with some pendant edges.
Therefore, $G$ consists of the triangle $C$ with some pendant edges attached to it.
We can thus find a linear $3$-coloring of $G$, by giving the same color to $v_i$ and the leaves adjacent to $v_{i+1}$ for each $i\in \{1,2,3\}$ (indices modulo $3$). 
This contradicts the fact that $\clin(G)= 4$.
\end{proof}

The above claim implies that there is at most one good vertex of $C$.
To conclude the proof it is enough to observe that, in the case of a single good vertex, we obtain an $F_2$, while in the case of no good vertex we obtain an $F_8$, a contradiction.

This completes the proof.
\end{proof}

\section{Algorithmic aspects}
\label{sec:algo}

We now discuss the complexity of \textsc{Linear Chromatic Number}, the decision problem that takes as input a graph $G$ and an integer $k$, and the task is to determine if $\clin(G)\le k$.
It is not clear if the problem is in \textsf{NP}, since, as shown by Kun et al.~\cite{kun2021polynomial}, the problem of determining whether a coloring is linear is \textsf{co-NP}-complete. 
Nevertheless, the close relationship between the centered and linear chromatic numbers, together with known hardness results for the centered chromatic number, lead to the following result.
A graph $G$ is \emph{cobipartite} if its vertex set is a union of two cliques.

\begin{theorem}\label{th:npc}
\textsc{Linear Chromatic Number} is \textsf{NP}-complete for cobipartite graphs.
\end{theorem}

\begin{proof}
Every connected cobipartite graph has a Hamiltonian path and every connected induced subgraph of a cobipartite graph is cobipartite.
Thus, by \Cref{obs: Hamiltonian}, any linear coloring of a connected cobipartite graph $G$ is also a centered coloring, hence, $\ccen(G) = \clin(G)$.
This implies that \textsc{Linear Chromatic Number}, when restricted to cobipartite graphs, is in \textsf{NP}: if $\clin(G)\le k$, then this can be certified by giving a treedepth decomposition with depth at most $k$, which is an equivalent definition of centered chromatic number (see \cite{nevsetvril2012sparsity}).
As shown by Bodlaender et al.~\cite{MR1612885}, given a cobipartite graph $G$ and an integer $k$, it is \textsf{NP}-complete to determine if $\ccen(G)\le k$.
Since $\ccen(G)\le k$ if and only if $\clin(G)\le k$, the claimed \textsf{NP}-completeness follows.
\end{proof}

On the positive side, we observe that the problem admits a linear-time fixed-parameter tractable algorithm with respect to its natural parameterization.

\begin{theorem}\label{th:fpt}
For every positive integer $k$, there is an algorithm running in time $\Oh(n)$ that determines if a given $n$-vertex graph $G$ satisfies $\clin(G)\le k$.
\end{theorem}

\begin{proof}
Before describing the algorithm, let us mention two observations that are crucial for the algorithm and its correctness.
First, if $G$ is a graph with $\clin(G)\le k$, then $\td(G) = \ccen(G) = \Oh(k^{11})$ (by \Cref{th:polybd} with the improved bound of \cite{bose2022linear}), and consequently $\tw(G)= \Oh(k^{11})$ (since any graph $G$ satisfies \hbox{$\tw(G)\le \td(G)-1$}).
Second, for a fixed $k$, the property that a graph $G=(V,E)$ admits a linear coloring with $k$ colors can be expressed with a fixed \textsf{MSO}$_2$ formula $\psi_k$.
The formula states that there exists a partition of $V$ into pairwise disjoint subsets $X_1,\ldots, X_k$ (the $k$~color classes) such that for each nonempty subset $F\subseteq E$ forming a path there exists a vertex $v$ incident with an edge of $F$ such that $v\in X_i$ for some $i\in \{1,\ldots, k\}$ and no other vertex incident with an edge of $F$ belongs to $X_i$.
Expressing that $F$ forms a path can be done by requiring two conditions on the subgraph formed by $F$: (i) that it is connected (that is, that for each partition of the set of vertices incident with an edge in $F$ into two nonempty parts there exists an edge in $F$ connecting the two parts), and (ii) that all its vertices have degree~$2$, except two of them, which have degree one. 

Let $G$ be a graph and let $n$ be the order of $G$.
The algorithm to determine if $\clin(G)\le k$ is now easy to obtain.
First, we apply any of the algorithms from~\cite{MR3479705} or \cite{Korhonen2023} to determine in time $2^{\Oh(k^{11})}n$ if $\tw(G)= \Oh(k^{11})$.
If this is not the case, we conclude that $\clin(G)> k$.
If $\tw(G)= \Oh(k^{11})$, then we also obtain a tree decomposition of $G$ with width $\Oh(k^{11})$.
Next, given this tree decomposition, we apply Courcelle's theorem~\cite{MR1042649} to determine in time $\Oh(n)$ (where the hidden constant depends on $k$) if $G$ models the formula $\psi_k$.
Since this is the case if and only if $G$ admits a linear coloring with $k$ colors, the result follows.
\end{proof}

\section{Conclusion}
\label{sec:concl}

In this paper we explored various aspects of the linear chromatic number of graphs. 
In particular, we gave improved bounds for several graph classes. 
We recall that most of our results in this direction are listed in \Cref{table:sum}. 
We investigated the subgraph obstructions to bounded linear chromatic number and provided two algorithmic results: an \textsf{NP}-completeness proof and an FPT algorithm.
So far, the conjecture that motivated this research is still open. Let us restate it below.

\conj*

As noted in the introduction, the ratio 2 above cannot be improved in general. This suggests the following open problem.

\begin{question}
    What is the supremum $c$ of $\ccen(T)/\clin(T)$ when $T$ is a tree?
\end{question}
By \Cref{th:log3}, $c\geq \log 3$.

\paragraph{Acknowledgements.}
The authors thank Jana Masa\v{r}\'ikov\'a and Wojciech Nadara for allowing them to include \Cref{prop:mn} and its proof.
This work was supported in part by the Slovenian Research and Innovation Agency (I0-0035, research program P1-0285 and research projects J1-3003, J1-4008, J1-4084, J1-60012, J1-70035, J1-70046, J5-4596, and N1-0370), and by the research program CogniCom (0013103) at the University of Primorska.
J.-F.~Raymond was partially supported by the ANR project GRALMECO (ANR-21-CE48-0004).

\bibliographystyle{alpha}
\bibliography{clin}

@article{kuhn2025computing,
  title={Computing Treedepth Obstructions},
  author={K{\"u}hn, Kolja},
  journal={arXiv preprint arXiv:2512.01658},
  year={2025}
}

@article{kun2021polynomial,
  title={Polynomial treedepth bounds in linear colorings},
  author={Kun, Jeremy and O’Brien, Michael P and Pilipczuk, Marcin and Sullivan, Blair D},
  journal={Algorithmica},
  volume={83},
  number={1},
  pages={361--386},
  year={2021},
  publisher={Springer}
}

@article{Korhonen2023,
author = {Korhonen, Tuukka},
title = {A Single-Exponential Time 2-Approximation Algorithm for Treewidth},
journal = {SIAM Journal on Computing},
note = {To appear. \url{https://doi.org/10.1137/22M147551X}},
doi = {10.1137/22M147551X},
URL = {https://doi.org/10.1137/22M147551X
},
eprint = {https://doi.org/10.1137/22M147551X},
year={2023}
}

@article {MR1042649,
    AUTHOR = {Courcelle, Bruno},
     TITLE = {The monadic second-order logic of graphs. {I}. {R}ecognizable
              sets of finite graphs},
   JOURNAL = {Inform. and Comput.},
  FJOURNAL = {Information and Computation},
    VOLUME = {85},
      YEAR = {1990},
    NUMBER = {1},
     PAGES = {12--75},
      ISSN = {0890-5401,1090-2651},
       DOI = {10.1016/0890-5401(90)90043-H},
       URL = {https://doi.org/10.1016/0890-5401(90)90043-H},
}

@article {MR3479705,
    AUTHOR = {Bodlaender, Hans L. and Drange, P\aa l Gr{\o}n{\aa}s and Dregi,
              Markus S. and Fomin, Fedor V. and Lokshtanov, Daniel and
              Pilipczuk, Micha{\l}},
     TITLE = {A {$c^kn$} 5-approximation algorithm for treewidth},
   JOURNAL = {SIAM J. Comput.},
  FJOURNAL = {SIAM Journal on Computing},
    VOLUME = {45},
      YEAR = {2016},
    NUMBER = {2},
     PAGES = {317--378},
      ISSN = {0097-5397,1095-7111},
       DOI = {10.1137/130947374},
       URL = {https://doi.org/10.1137/130947374},
}

@article {MR1612885,
    AUTHOR = {Bodlaender, Hans L. and Deogun, Jitender S. and Jansen, Klaus
              and Kloks, Ton and Kratsch, Dieter and M\"{u}ller, Haiko and
              Tuza, Zsolt},
     TITLE = {Rankings of graphs},
   JOURNAL = {SIAM J. Discrete Math.},
  FJOURNAL = {SIAM Journal on Discrete Mathematics},
    VOLUME = {11},
      YEAR = {1998},
    NUMBER = {1},
     PAGES = {168--181},
      ISSN = {0895-4801,1095-7146},
   MRCLASS = {68R10 (05C15 05C78 05C85)},
  MRNUMBER = {1612885},
MRREVIEWER = {Peter\ Damaschke},
       DOI = {10.1137/S0895480195282550},
       URL = {https://doi.org/10.1137/S0895480195282550},
}

@article {MR0947254,
    AUTHOR = {Olariu, Stephan},
     TITLE = {Paw-free graphs},
   JOURNAL = {Inform. Process. Lett.},
  FJOURNAL = {Information Processing Letters},
    VOLUME = {28},
      YEAR = {1988},
    NUMBER = {1},
     PAGES = {53--54},
      ISSN = {0020-0190,1872-6119},
   MRCLASS = {05C35 (68Q25 68R10)},
  MRNUMBER = {947254},
MRREVIEWER = {Narsingh\ Deo},
       DOI = {10.1016/0020-0190(88)90143-3},
       URL = {https://doi.org/10.1016/0020-0190(88)90143-3},
}

@phdthesis{pothen1988complexity,
  title={The complexity of optimal elimination trees},
  author={Pothen, Alex},
  school={Pennsylvania State
University, Department of Computer Science},
  year={1988}
}

@article{katchalski1995ordered,
  title={Ordered colourings},
  author={Katchalski, Meir and McCuaig, William and Seager, Suzanne},
  journal={Discrete Mathematics},
  volume={142},
  number={1-3},
  pages={141--154},
  year={1995},
  publisher={Elsevier}
}

@article{schaffer1989optimal,
  title={Optimal node ranking of trees in linear time},
  author={Sch{\"a}ffer, Alejandro A},
  journal={Information Processing Letters},
  volume={33},
  number={2},
  pages={91--96},
  year={1989},
  publisher={Elsevier}
}

@book{nevsetvril2012sparsity,
  title={Sparsity: graphs, structures, and algorithms},
  author={Ne{\v{s}}et{\v{r}}il, Jaroslav and Ossona de Mendez, Patrice},
  volume={28},
  year={2012},
  publisher={Springer Science \& Business Media}
}

@article{bose2022linear,
  title={Linear versus centred chromatic numbers},
  author={Bose, Prosenjit and Dujmovi{\'c}, Vida and Houdrouge, Hussein and Javarsineh, Mehrnoosh and Morin, Pat},
  journal={\textrm{Preprint available at \url{https://arxiv.org/abs/2205.15096}}},
  year={2022}
}

@article {dvorak2012forbidden,
    AUTHOR = {Dvo\v{r}\'{a}k, Zden\v{e}k and Giannopoulou, Archontia C. and
              Thilikos, Dimitrios M.},
     TITLE = {Forbidden graphs for tree-depth},
   JOURNAL = {European J. Combin.},
  FJOURNAL = {European Journal of Combinatorics},
    VOLUME = {33},
      YEAR = {2012},
    NUMBER = {5},
     PAGES = {969--979},
       DOI = {10.1016/j.ejc.2011.09.014},
}

@article{czerwinski2021improved,
  title={Improved bounds for the excluded-minor approximation of treedepth},
  author={Czerwinski, Wojciech and Nadara, Wojciech and Pilipczuk, Marcin},
  journal={SIAM Journal on Discrete Mathematics},
  volume={35},
  number={2},
  pages={934--947},
  year={2021},
  publisher={SIAM}
}

@inproceedings{DBLP:conf/innovations/Lee17,
  author       = {James R. Lee},
  editor       = {Christos H. Papadimitriou},
  title        = {Separators in Region Intersection Graphs},
  booktitle    = {8th Innovations in Theoretical Computer Science Conference, {ITCS}
                  2017, January 9-11, 2017, Berkeley, CA, {USA}},
  series       = {LIPIcs},
  volume       = {67},
  pages        = {1:1--1:8},
  publisher    = {Schloss Dagstuhl - Leibniz-Zentrum f{\"{u}}r Informatik},
  year         = {2017},
  url          = {https://doi.org/10.4230/LIPIcs.ITCS.2017.1},
}

@inproceedings {MR1642971,
    AUTHOR = {Bodlaender, Hans and Gustedt, Jens and Telle, Jan Arne},
     TITLE = {Linear-time register allocation for a fixed number of
              registers},
 BOOKTITLE = {Proceedings of the {N}inth {A}nnual {ACM}-{SIAM} {S}ymposium
              on {D}iscrete {A}lgorithms ({S}an {F}rancisco, {CA}, 1998)},
     PAGES = {574--583},
 PUBLISHER = {ACM, New York},
      YEAR = {1998},
      ISBN = {0-89871-410-9},
   MRCLASS = {68N20},
  MRNUMBER = {1642971},
}

@incollection {MR0634535,
    AUTHOR = {Duffus, Dwight and Jacobson, Michael Scott and Gould, Ronald J.},
     TITLE = {Forbidden subgraphs and the {H}amiltonian theme},
 BOOKTITLE = {The theory and applications of graphs ({K}alamazoo, {M}ich.,
              1980)},
     PAGES = {297--316},
 PUBLISHER = {Wiley, New York},
      YEAR = {1981},
      ISBN = {0-471-08473-5},
   MRCLASS = {05C45 (05C38)},
  MRNUMBER = {634535},
MRREVIEWER = {H.\ A.\ Jung},
}

@article {Gavril1974,
    AUTHOR = {Gavril, F\u{a}nic\u{a}},
     TITLE = {The intersection graphs of subtrees in trees are exactly the
              chordal graphs},
   JOURNAL = {J. Combinatorial Theory Ser. B},
  FJOURNAL = {Journal of Combinatorial Theory. Series B},
    VOLUME = {16},
      YEAR = {1974},
     PAGES = {47--56},
      ISSN = {0095-8956},
   MRCLASS = {05C05},
  MRNUMBER = {332541},
MRREVIEWER = {L. W. Beineke},
       DOI = {10.1016/0095-8956(74)90094-x},
       URL = {https://doi.org/10.1016/0095-8956(74)90094-x},
}

@article {Walter1978,
    AUTHOR = {Walter, James R.},
     TITLE = {Representations of chordal graphs as subtrees of a tree},
   JOURNAL = {J. Graph Theory},
  FJOURNAL = {Journal of Graph Theory},
    VOLUME = {2},
      YEAR = {1978},
    NUMBER = {3},
     PAGES = {265--267},
      ISSN = {0364-9024},
   MRCLASS = {05C99},
  MRNUMBER = {505894},
MRREVIEWER = {Jerald A. Kabell},
       DOI = {10.1002/jgt.3190020311},
       URL = {https://doi.org/10.1002/jgt.3190020311},
}

@article{Buneman1974,
    AUTHOR = {Buneman, Peter},
     TITLE = {A characterisation of rigid circuit graphs},
   JOURNAL = {Discrete Math.},
  FJOURNAL = {Discrete Mathematics},
    VOLUME = {9},
      YEAR = {1974},
     PAGES = {205--212},
      ISSN = {0012-365X},
   MRCLASS = {05C35},
  MRNUMBER = {357218},
MRREVIEWER = {L. W. Beineke},
       DOI = {10.1016/0012-365X(74)90002-8},
       URL = {https://doi.org/10.1016/0012-365X(74)90002-8},
}

@article{demaine2008linearity,
  title={Linearity of grid minors in treewidth with applications through bidimensionality},
  author={Demaine, Erik D and Hajiaghayi, MohammadTaghi},
  journal={Combinatorica},
  volume={28},
  pages={19--36},
  year={2008},
  publisher={Springer}
}

@article{DRRSSS14,
    AUTHOR = {Demaine, Erik D. and Reidl, Felix and Rossmanith, Peter and
              S\'anchez Villaamil, Fernando and Sikdar, Somnath and
              Sullivan, Blair D.},
     TITLE = {Structural sparsity of complex networks: bounded expansion in
              random models and real-world graphs},
   JOURNAL = {J. Comput. System Sci.},
  FJOURNAL = {Journal of Computer and System Sciences},
    VOLUME = {105},
      YEAR = {2019},
     PAGES = {199--241},
      ISSN = {0022-0000,1090-2724},
   MRCLASS = {05C82 (68W40)},
  MRNUMBER = {3975051},
       DOI = {10.1016/j.jcss.2019.05.004},
       URL = {https://doi.org/10.1016/j.jcss.2019.05.004},
}

@article{DKT13,
  author       = {Zdenek Dvor{\'{a}}k and
                  Daniel Kr{\'{a}}l and
                  Robin Thomas},
  title        = {Testing first-order properties for subclasses of sparse graphs},
  journal      = {J. {ACM}},
  volume       = {60},
  number       = {5},
  pages        = {36:1--36:24},
  year         = {2013},
  url          = {https://doi.org/10.1145/2499483},
  doi          = {10.1145/2499483},
  bibsource    = {dblp computer science bibliography, https://dblp.org}
}

@article{NO08,
  author       = {Jaroslav Nesetril and
                  Patrice Ossona de Mendez},
  title        = {Grad and classes with bounded expansion {II.} Algorithmic aspects},
  journal      = {Eur. J. Comb.},
  volume       = {29},
  number       = {3},
  pages        = {777--791},
  year         = {2008},
  url          = {https://doi.org/10.1016/j.ejc.2006.07.014},
  doi          = {10.1016/J.EJC.2006.07.014},
  bibsource    = {dblp computer science bibliography, https://dblp.org}
}

@book {MR4874150,
    AUTHOR = {Diestel, Reinhard},
     TITLE = {Graph theory},
    SERIES = {Graduate Texts in Mathematics},
    VOLUME = {173},
   EDITION = {Sixth},
 PUBLISHER = {Springer, Berlin},
      YEAR = {[2025] \copyright 2025},
     PAGES = {xx+454},
      ISBN = {978-3-662-70106-5; 978-3-662-70107-2},
   MRCLASS = {05-01 (05Cxx)},
  MRNUMBER = {4874150},
}

\appendix
\section{Pseudogrids}\label{sec:pseudogrids}

Recall that pseudogrids are the subgraph-minimal graphs that contain $\grid_k$ for some integer $k$. The following lemma corresponds to the result given in \cite{bose2022linear} on pseudogrids. However, their paper focuses on the lower bound only. We give here a proof for the upper bound. 

\begin{lemma}\label{lem:grids-and-pseudogrids}
There are constants $c,c'>0$ such that for every $k>1$, \[ck \leq \clin(\grid_k) \leq \ccen(\grid_k) \leq c' k.\]
Furthermore, there is a constant $c''>0$ such that for every $k>1$ and every subgraph-minimal graph $G$ containing $\grid_k$ as a minor (i.e., $k\times k$ pseudogrid) we have $\ccen(G)/\clin(G)\leq c''$.
\end{lemma}

\begin{proof}
    The lower bound is \Cref{th:lineargrid}.
    The upper bound can be obtained by induction on $k$ as follows, with $c' = 4$. 
    The base case $k=2$ is trivial since the two coloring numbers are equal. 
    When $k>2$, we pick a set of vertices $X$ consisting of the middle row and the middle column of $\grid_k$ (if $k$ is even, we choose one of the two rows and columns  arbitrarily). This set contains $2k-1$ vertices, to which we assign a unique color each.
    Observe that each connected component of $\grid_k-X$ is a copy of $\grid_{k'}$ for some $k' \leq (k-1)/2$. By induction, there is a centered coloring of each such component with at most $c'k'$ colors. Fix such a coloring, with the same set of colors of each of the 4 components. We claim that this, together with the aforementioned coloring of $X$, is a centered coloring of $\grid_k$ with at most $c'k$ colors. The colors bound holds since we use at most $2k$ colors for $X$ and at most $c'k/2$ for $\grid_k-X$, so at most $c'k$ in total as $c'=4$. It is a centered coloring because any subgraph of $\grid_k$ either is included in a single component of $G-X$, in which case it has a uniquely colored vertex, by induction, or the subgraph contains a vertex of $X$, which has a unique color by definition.
    
    We now prove the second part of the statement. 
    We may assume that $k\ge 3$.
    By \Cref{th:lineargrid} we have $\clin(G)\geq ck$.  
    Here we will use the following equivalent (see, e.g., \cite{MR4874150}) definition of a minor: $H'$ is a minor of $H$ if there is a function mapping vertices of $H'$ to disjoint subsets of $H$, each inducing a connected subgraph, and such that adjacent vertices of $H'$ are mapped to subsets that are connected by an edge in $G$. So in the case we study, let $f$ be such a mapping from $V(\grid_k)$ to vertex subsets of $G$.

    Note that because $G$ is subgraph-minimal and $\grid_k$ has maximum degree 4, every vertex has degree 4, 3, or 2 in $G$. For the same reason $\{f(v) \mid v\in V(\grid_k)\}$ is a partition of $V(G)$ and each such set $f(v)$ induces a tree and has at most two vertices that have degree more than $2$ in $G$.
    
    Suppose first that for every $v\in V(\grid_k)$, the set $f(v)$ contains at most two vertices.
    Then we can use a coloring of $\grid_k$ with $c'k$ colors as described above and duplicate each color in order to color the vertices in $G$. 
    For instance, if $v\in V(\grid_k)$ receives some color $x$, we assign colors $(0,x)$ and $(1,x)$ to the two vertices of $f(v)$. 
    As noted above, the image of $f$ defines a partition of $V(G)$. 
    This shows that in this case $\ccen(G) \leq 2c'k$ and we are done.

    Let us now consider the case where the sizes of the images of $f$ are not restricted.
    Let $W$ be the graph obtained from $G$, iteratively, as follows: as long as $|f(v)| > 2$ for some $v\in V(\grid_k)$, we arbitrarily pick a vertex of $f(v)$ that has degree~$2$ in $G$, delete it, and join its neighbors by an edge (if they were not already adjacent). 
    Such a vertex exists in this case by the above remark about subgraph-minimality.
    When this process ends, the resulting graph $W$ satisfies the conditions of the first case so we can color it with $2c'k$ colors.

    Notice that the vertices we deleted in the above process induce a collection of paths in $G$, the vertices of which have degree 2 in $G$.
    Let $s$ be the maximum order of a path of $G$ whose vertices have degree $2$ in $G$. 
    As $\clin(P_s)\geq \ceil{\log s}$ by \cref{lem:tdpk}, we have $\clin(G) \geq \ceil{\log s}$.

    Then we can get a centered coloring of $G$ as follows: the vertices that also exist in $W$ receive the same color as in this graph and each maximal path of $G-V(W)$ is colored with the same set of at most $\ceil{\log s}$ colors as in \cref{lem:tdpk}. 
    This is indeed a centered coloring, as every subgraph of $G$ either is a subgraph of one of the aforementioned paths of degree 2 vertices, in which case it has a center by the choice of the colors in this path, or it contains a vertex of $W$, and the center exists by the choice of the coloring of $W$.
    We get $\ccen(G) \leq 2c'k + \ceil{\log s}$.
    As noted above $\clin(G) \geq ck$ and $\clin(G) \geq \ceil{\log s}$, so the ratio $\ccen(G)/\clin(G)$ is bounded from above by a constant.
\end{proof}

\section{Obstructions for linear chromatic number 3}\label{appendix-small-values}
\setlength{\extrarowheight}{5pt}

\begin{lemma}\label{lem:obst}
Let $F\in \mathcal{F}$ be one of the graphs depicted in \Cref{fig:listObs}. 
Then $\clin(F) \ge 4$.
\end{lemma}

\begin{proof}
If $F$ is in $\{F_1 , F_3 , F_6 , F_8 , K_4 , C_5 , C_6 , C_7\}$, then $F$ is (claw, net)-free, hence, by \Cref{th:CN-free}, $\clin(F) = \ccen(F) \ge 4$, where the inequality follows from \cite[Theorem~4]{dvorak2012forbidden}).
For $P_8$ this follows from \Cref{lem:tdpk}.

In the following we handle the remaining graphs of \Cref{fig:listObs}, namely $F_2$, $F_4$, $F_5$, $F_7$, and $F_9$ in order, by assuming there exists a linear coloring $\varphi$ with colors $\{0,1,2\}$ and reaching a contradiction. 

\bigskip

\noindent
{\small
\begin{tabular}{ | c  m{12cm}  | }
	 \hline
	\begin{minipage}{.2\textwidth}
		\includegraphics[scale=1.4]{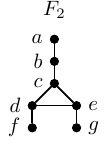}
	\end{minipage}
	&
	For the graph $F_2$, labeled as in the figure on the left, we can make the following successive assumptions and implications.
	\begin{enumerate}
		\setlength{\itemsep}{2pt}
		 \setlength{\parskip}{0pt}
		\item $\varphi(c) = 0$, $\varphi(d) =1$, and $\varphi(e) =2$ because the vertices form a triangle and we can choose the values by symmetry;
		\item $\varphi(f) =0$ or $\varphi(g)= 0$ because of the path  $\mathit{fdeg}$; suppose without loss of generality that $\varphi(f) =0$;
		\item $\varphi(b) = 2$ because of the path $\mathit{fdcb}$;
		\item $\varphi(a) =1$ because of the path $\mathit{abce}$.
	\end{enumerate}
	Then, the path $\mathit{fdecba}$ has no center, a contradiction.
	 \vspace{5pt}   \\ 
	 \hline
	
	
	\begin{minipage}{.2\textwidth}
		\includegraphics[scale=1.4]{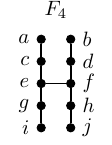}
	\end{minipage}
	&
	Consider now the graph $F_4$, labeled as in the figure on the left. 
    We refer to the components of $F_4-\{e,f\}$ as the top/bottom left/right component, according to their position on the picture. We can make the following successive assumptions and implications.
	\begin{enumerate}
			\setlength{\itemsep}{2pt}
		\setlength{\parskip}{0pt}
		\item $\varphi(e) =0$ and $\varphi(f)=1$, by symmetry;
		\item Color 0 appears in one of the right components of $F_4-\{e,f\}$, otherwise $F_4-\{e\}$ would contain a $P_5$ colored with two colors (and, hence, without a center). 
		We can assume without of loss of generality that 0 appears in the top right component, and symmetrically that 1 appears in the top left component.
		\item Any component of $F_4-\{e,f\}$ has a vertex colored 2, otherwise it would form, together with $e$ and $f$, a path with no center.        
	\end{enumerate}
	Then the path formed by the two top components and vertices $u$ and $v$ contains twice each color, hence, this path has no center, a contradiction.
	\vspace{5pt}   \\  \hline
	\begin{minipage}{.2\textwidth}
		\includegraphics[scale=1.5]{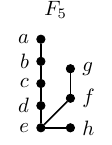}
	\end{minipage}
	&
	For the graph $F_5$, labeled as in the figure on the left, note that $\textit{abcdefg}$ is a $P_7$, so it has to have a unique color on $d$, as any other choice of a center would leave a path on at least 4 vertices to be colored with 2 colors.
	We can make the following successive assumptions and implications.
	\begin{enumerate}
			\setlength{\itemsep}{2pt}
		\setlength{\parskip}{0pt}
		\item $\varphi(d)=0$ (by symmetry of the colors) and color 0 does not appear on the paths $\mathit{abc}$ and $\mathit{efg}$;
		\item $\varphi(h) =0$ because of the path  $\mathit{hefg}$ and since color 0 does not appear on $\mathit{efg}$;
		\item $\varphi(e) = 1$ by symmetry of the colors 1 and 2;
		\item $\varphi(c) = 2$ because of the path $\mathit{hedc}$;
		\item $\varphi(b)=1$ and $\varphi(a)=2$ since color 0 does not appear on the path $\mathit{abc}$.
	\end{enumerate}
	Then, the path $\mathit{abcdeh}$ has no center, a contradiction.\vspace{5pt}   \\   \hline
    \end{tabular}

        \newpage
\begin{tabular}{ | c  m{12cm}  | }
	 \hline
		\begin{minipage}{.2\textwidth}
			\includegraphics[scale=1.4]{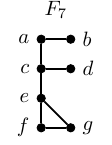}
		\end{minipage}
		&
		For the graph $F_7$, labeled as in the figure on the left, we can make the following successive assumptions and implications.
		\begin{enumerate}
				\setlength{\itemsep}{2pt}
			\setlength{\parskip}{0pt}
			\item $\varphi(e) = 0$, $\varphi(f) = 1$, $\varphi(g) = 2$ because the vertices form a triangle and we can choose the values by symmetry;
			\item $\varphi(c) = 1$ by symmetry of the colors 1 and 2; 
			\item $\varphi(a) = \varphi(d) = 2$ because of the paths $\mathit{feca}$ and $\mathit{fecd}$, respectively;
			\item $\varphi(b) = 0$ because of the path $\mathit{bacd}$.
		\end{enumerate}
		Then, the path $\mathit{bacefg}$ has no center, a contradiction.
		\vspace{5pt}   \\  \hline	

    \begin{minipage}{.2\textwidth}
			\vspace{10pt}
			\includegraphics[scale=1.4]{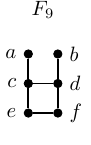}
	\end{minipage}
            &
		Finally, consider the graph $F_9$, labeled as in the figure on the left. 
        The three colors appear on the cycle $\mathit{cdfe}$, one of them appears twice on non-adjacent vertices of that $C_4$. By symmetry of the graph and the colors, we can assume $\varphi(c)=\varphi(f) = 0$, $\varphi(d) = 1$, and $\varphi(e) = 2$.
		Therefore $\varphi(a)= 2$ because of the path $\mathit{acdf}$.
		Then, the path $\mathit{acef}$has no center, a contradiction. \\ \hline
	\end{tabular}
}\\[15pt]
This concludes the proof of \cref{lem:obst}.
\end{proof}

\end{document}